\documentclass[review]{siamart1116}

\usepackage{lipsum}
\usepackage{amsfonts}
\usepackage{graphicx}
\usepackage{epstopdf}
\usepackage{algorithmic}
\ifpdf
  \DeclareGraphicsExtensions{.eps,.pdf,.png,.jpg}
\else
  \DeclareGraphicsExtensions{.eps}
\fi

\numberwithin{theorem}{section}

\newcommand{\TheTitle}{Construction of h-refined continuous finite element spaces with arbitrary hanging node configurations 
and applications to multigrid algorithms} 
\newcommand{\TheAuthors}{Eugenio Aulisa, Giacomo Capodaglio, Guoyi Ke}

\headers{\TheTitle}{\TheAuthors}

\title{{\TheTitle}\thanks{Submitted to the editors October 2, 2017.
\funding{This work was supported by the National Science Foundation grant DMS-1412796.}}}

\author{
  Eugenio Aulisa\thanks{Broadway \& Boston, Department of Mathematics and Statistics, Texas Tech University, Lubbock TX 79409, USA
    (\email{eugenio.aulisa@ttu.edu}),
    (\email{giacomo.capodaglio@ttu.edu}),
    (\email{guoyi.ke@ttu.edu}).}
  \and
  Giacomo Capodaglio \footnotemark[2]
  \and
  Guoyi Ke \footnotemark[2]
}

\headers{h-refined finite element spaces with hanging nodes}{Eugenio Aulisa, Giacomo Capodaglio, Guoyi Ke}

\usepackage{amsopn}

\ifpdf
\hypersetup{
  pdftitle={\TheTitle},
  pdfauthor={\TheAuthors}
}
\fi

\allowdisplaybreaks
\expandafter\def\expandafter\normalsize\expandafter{%
    \normalsize
    \setlength\abovedisplayskip{5pt}
    \setlength\belowdisplayskip{5pt}
    \setlength\abovedisplayshortskip{5pt}
    \setlength\belowdisplayshortskip{5pt}
}

\usepackage[english]{babel}
\usepackage{amssymb}
\usepackage{amsmath, bm}
\usepackage{latexsym}
\usepackage{fancyhdr}
\usepackage{graphicx}
\usepackage{newlfont}
\usepackage{textcomp}
\usepackage{float}
\usepackage{multirow}

\usepackage{enumitem}

\newtheorem{remark}{Remark}

\usepackage{color}
\renewcommand*{\vec}[1]{\mathbf{#1}}

\newcommand{\new}[1]{\textcolor{black}{{#1}}}
\newcommand{\old}[1]{\textcolor{red}{}}

\begin{document}

\maketitle 

\begin{abstract}
We present a novel approach for the construction of basis functions to be employed
in selective or adaptive h-refined finite element applications with arbitrary-level hanging node configurations.
Our analysis is not restricted to $1$-irregular meshes, as it is usually done in the literature, 
allowing our results to be applicable to a broader class of local refinement strategies.
The proposed method does not require the solution of any linear system to obtain the constraints necessary 
to enforce continuity of the basis functions and it can be easily implemented.
A mathematical analysis is carried out to prove that the proposed basis functions are continuous and linearly independent.
Finite element spaces are then defined as the spanning sets of such functions,
and the implementation of a multigrid algorithm built on these spaces is discussed.
A spectral analysis of the multigrid algorithm 
highlights superior convergence properties of the proposed method 
over existing strategies based on a local smoothing procedure.
Finally, linear and nonlinear numerical examples 
are tested to show the robustness and versatility of the multigrid algorithm. 
\end{abstract}



\section{Introduction}
In many scientific applications, the problem of obtaining an accurate solution in a timely manner
is of major relevance. In a finite element setting, this could be done in a variety of ways.
For instance, local refinement of the finite element grid allows introducing new degrees of freedom only in the areas
of the domain where a more accurate solution is sought.
In this way, the addition of extra degrees of freedom on the entire domain can be avoided, with the result that the computational
time can be significantly decreased.
The method of locally refining the mesh to obtain a more accurate solution on prescribed parts of the domain is usually referred to as h-refinement.
Alternatively, a dual idea consists of increasing the degree of the polynomial functions that approximate the solution only on the elements that require greater accuracy. This approach is usually referred to as p-refinement.
With the intention of combining the advantages of h-refinement and p-refinement, mixed methods have also been introduced \cite{babuvska1981error, guo1986hp, bangerth2009data},
known as hp-refinement strategies.

\new{In this work, the analysis is set in the framework of h-refinement for Lagrangian finite elements and local midpoint refinement is adopted as refinement strategy.}
With this choice, special nodes called hanging nodes are introduced \cite{carstensen2009hanging} and the finite element solution ceases to be continuous, unless extra care is exerted to prevent this.
For instance, continuity could be enforced using a multilevel approach \cite{zander2015multi, aulisa2017mgdd}.
More frequently, modified basis functions for the finite element spaces are considered.
As pointed out in \cite{fries2011hanging}, basis functions may be altered in two different ways, depending on whether or not
degrees of freedom are associated with the hanging nodes.
Our analysis fits in the framework of {\it constrained approximation}, where
no degrees of freedom are associated with the hanging nodes.
We refer to \cite{gupta1978finite, morton1995new} for works that, on the contrary, assign degrees of freedom to the hanging nodes.

In constrained approximation, most of the studies available in the literature require the finite element grids to satisfy the so called $1$-irregularity condition \cite{bank1983some, rachowicz1989toward, rachowicz2000hp, rachowicz2002hp, vsolin2004goal, schroder2011constrained, demkowicz1989toward}.
Such a condition requires the maximum difference between the refinement level of two adjacent elements to be one \cite{vsolin2008arbitrary}.
In this work, we present a novel approach for the construction of continuous basis functions defined on grids with arbitrary-level hanging node configurations. Indeed, the $1$-irregularity condition is not enforced.

Only few recent works avoided the $1$-regularity requirement \cite{vsolin2008arbitrary, di2016easy, schroder2011constrained, kus2014arbitrary}.
The strategy presented in this paper is simpler and more general than the ones suggested in the existing works.
The method here proposed does not require the solution of a linear system for the determination of the constraint coefficients 
like \cite{vsolin2008arbitrary, bangerth2009data}, and it does not require the shape functions to be of tensor product type like \cite{di2016easy}.
Rather, it can be easily implemented once the tree structure of the grid has been determined.
Moreover, our approach works for applications in two and three dimensions and for a wide variety of finite element types, 
unlike other methods that only apply to $2$D problems or to specific element types in $3$D.
\new{In the deal.II library \cite{bangerth2007deal}, the case of $1$-irregular meshes has been implemented in the framework of hp-refinement. There, the constraints are enforced in an edge-by-edge fashion 
and can be determined by solving a linear system that enforces continuity \cite{bangerth2009data}.
However, when dealing with arbitrary-level hanging nodes an edge-by-edge approach is not sufficient 
and the constraints extend to nodes belonging to neighboring elements.
This is intrinsic in the nature of arbitrary-level hanging node configurations and cannot be avoided.
}
To complete our results, we present a theoretical analysis that shows the new functions 
are continuous and linearly independent, and therefore form a basis for their spanning set.

In our approach, all the burden of dealing with the hanging nodes is embedded 
in building a projection operator between the original discontinuous space and the continuous one. 
This approach allows the implementation of the proposed method on existing finite element codes with very few modifications, 
and it has been implemented in FEMuS \cite{femus-web-page}, 
an open-source finite element C++ library built on top of
PETSc \cite{petsc-web-page}.
All our theoretical investigation is carried out in a multilevel setting.
The construction of nested continuous finite element spaces arising from the proposed continuous basis functions is also discussed.
These finite element spaces and corresponding projection operators are well suited for multigrid methods.
The continuity of our finite element spaces allows the multigrid level smoothing to be performed on
all degrees of freedom. This global smoothing guarantees an arbitrary improvement in the convergence properties
that is directly proportional to the number of smoothing iterations. 
Such a feature cannot be achieved when the smoothing is performed only 
on a subspace of the multigrid spaces, as it is done in many multigrid strategies 
for local refinement \cite{bramble1993new, bramble1991convergence, janssen2011adaptive}.

The paper is structured as follows. In Section \ref{Formulation}, the method for the construction of continuous basis functions is laid out. We introduce the inter-space operators for our multigrid algorithm, and describe the implementation of 
a solver that makes use of the continuous finite element spaces here defined.
In Section \ref{eigAnalysis}, we compute the spectral radius of the error operator associated with 
the multigrid algorithm applied to the bilinear form arising from Poisson equation. A comparison is made with a multigrid algorithm that carries out smoothing only on subspaces \cite{bramble1991convergence}.
This comparison shows better convergence properties for our algorithm for either fixed or increasing number of smoothing iterations. In Sections \ref{numex1} and \ref{numex2}, 
linear and nonlinear numerical experiments are shown, where our multigrid
algorithm is used as a preconditioner for other linear solvers.
In particular, the 2D Poisson problem and the 3D buoyancy driven flow are used as numerical tests
to showcase the suitability and broad range of applicability of our multigrid to multiphysics problems.

\section{Formulation}\label{Formulation}
In this section, the modified basis functions that define the finite element spaces involved in the formulation of the problem are described. We introduce the inter-space operators used in the multigrid context, and summarize the steps of a solver
based on the proposed theory.
\subsection{Preliminaries}
Let $J$ be a non-negative integer and let $k = 0, \ldots, J$. 
\new{Here, $J$ represents the maximum degree of local refinement, while $k$ refers to a generic refined level.}
Let $\Omega$ be a closed and bounded subset of $\mathbb{R}^n$, for $n=1,2,3$ and let $\mathcal T_0$ 
be a regular triangulation of size $h_0$ on $\Omega$. 
\new{With a slight abuse of terminology, we use the word triangulation even for grids that are not composed of triangular elements. As a matter of fact, many element types other than triangles are considered in the numerical simulations.
Let $\{\Theta_k\}_{k=0}^{J}$ be a collection of closed overlapping subsets of $\Omega$, such that
$$ \Theta_J \subseteq \Theta_{J-1} \subseteq \cdots \subseteq \Theta_0 \equiv \Omega.$$
 For any given $k = 1, \ldots, J$, we can define recursively an irregular triangulation $\mathcal{T}_k$ on $\Omega$
starting from the triangulation  $\mathcal{T}_{k-1}$ in the following way.
\begin{definition}\label{defT_k}
Let $\mathcal{T}_{0}$ be a quasi--uniform triangulation built on $\Theta_0$. Then, for $k = 1, \ldots, J$,
$\mathcal{T}_{k}$ is obtained by midpoint refinement of only the elements of $\mathcal{T}_{k-1}$
that lie on $\Theta_{k}$.
\end{definition}
We next define new sequences of subsets of $\Omega$.
These are introduced to set our formulation in a multilevel framework,
in order to easily apply our analysis to multigrid methods.
\begin{definition}\label{Omegadownup}
For all $l = 0,\dots, k$, with $k \le J$, we define the sets
$${\Omega}_k^l \equiv \left\{ 
\begin{array}{l l}
\overline{\Theta_l \setminus \Theta_{l+1}} & \mbox{ if } l < k \\
\Theta_k &\mbox{ if } l = k
\end{array}
\right. ,$$
and $\Gamma_k^{l,n} \equiv \Omega_k^l \cap \Omega_k^n$, to be the interface between any two of such sets .
 \end{definition}
To summarize, $J$ denotes the maximum degree of refinement, $k$ refers to the current refinement level considered,
while $l$ identifies the different subdomains of $\Omega$ at the current level $k$.
 \begin{definition}\label{Omegadowndown}
For all $l = 0,\dots, k$, with $k \le J$ define
$$\Omega_{k,0} \equiv \Omega^0_k, \qquad {\Omega}_{k,l} \equiv \bigcup_{m=0}^l \Omega^m_k, \qquad \Omega_{k} \equiv \Omega _{k,k}.$$
\end{definition}
\begin{remark}\label{Omega_nested}
Note that $\Omega _{k} = \Omega$ and, by definition, $\{{\Omega}_{k,l}\}_{l=0}^k$ is a sequence of nested sets.
This sequence will be fundamental in the analysis that follows to define the discontinuous basis functions used
for the construction of the continuous ones.
\end{remark}
By construction of $\mathcal{T}_k$, we have that, for a given $k$ level, $k+1$ regular triangulations are presented on $k+1$ different subsets of $\Omega$.
We formalize this in the following definition.
\begin{definition}
 For all $l=0, \ldots, k$, with $k \leq J$ we define $\mathcal{T}^l_k$ to be the regular triangulation composed of all the elements $T \in \mathcal{T}_k$ such that $T \cap int(\Omega_k^l) \neq \emptyset$.
\end{definition}
From this definition we have that $\mathcal{T}^l_k$ is a triangulation defined only on $\Omega_k^l$ and not on the entire domain
$\Omega$ as $\mathcal{T}_k$.
Triangulations can be defined on $\Omega_{k,l} \subseteq \Omega$ in the following way.
\begin{definition}
For all $l=0, \ldots, k$, with $k \leq J$, define
$$\mathcal{T}_{k,0} \equiv \mathcal{T}^0_k, \qquad \mathcal{T}_{k,l} \equiv \bigcup_{m=0}^l \mathcal{T}^m_k.$$
Then $\mathcal{T}_{k,l}$ is an irregular triangulation on $\Omega_{k,l}$, for all $l=1,\ldots,k$.
\end{definition}
Note that $\mathcal{T}_{k,k}=\mathcal{T}_{k}$.
Moreover, the above definition is consistent with Definition \ref{Omegadowndown}.
}

Due to the local midpoint refinement procedure, hanging nodes are present on  $\mathcal{T}_k$, for $k=1, \ldots, J$.
These are special nodes of
some element $T_i \in \mathcal{T}_k$ that
lie on edges (or faces) of another element $T_j \in \mathcal{T}_k$ without being nodes for $T_j$.
For example, as it can be seen from Figure \ref{irregular_grid}, where linear elements are used, 
the hanging nodes are marked with black squares.
\begin{figure}[h]
\centering
\includegraphics[scale=0.32]{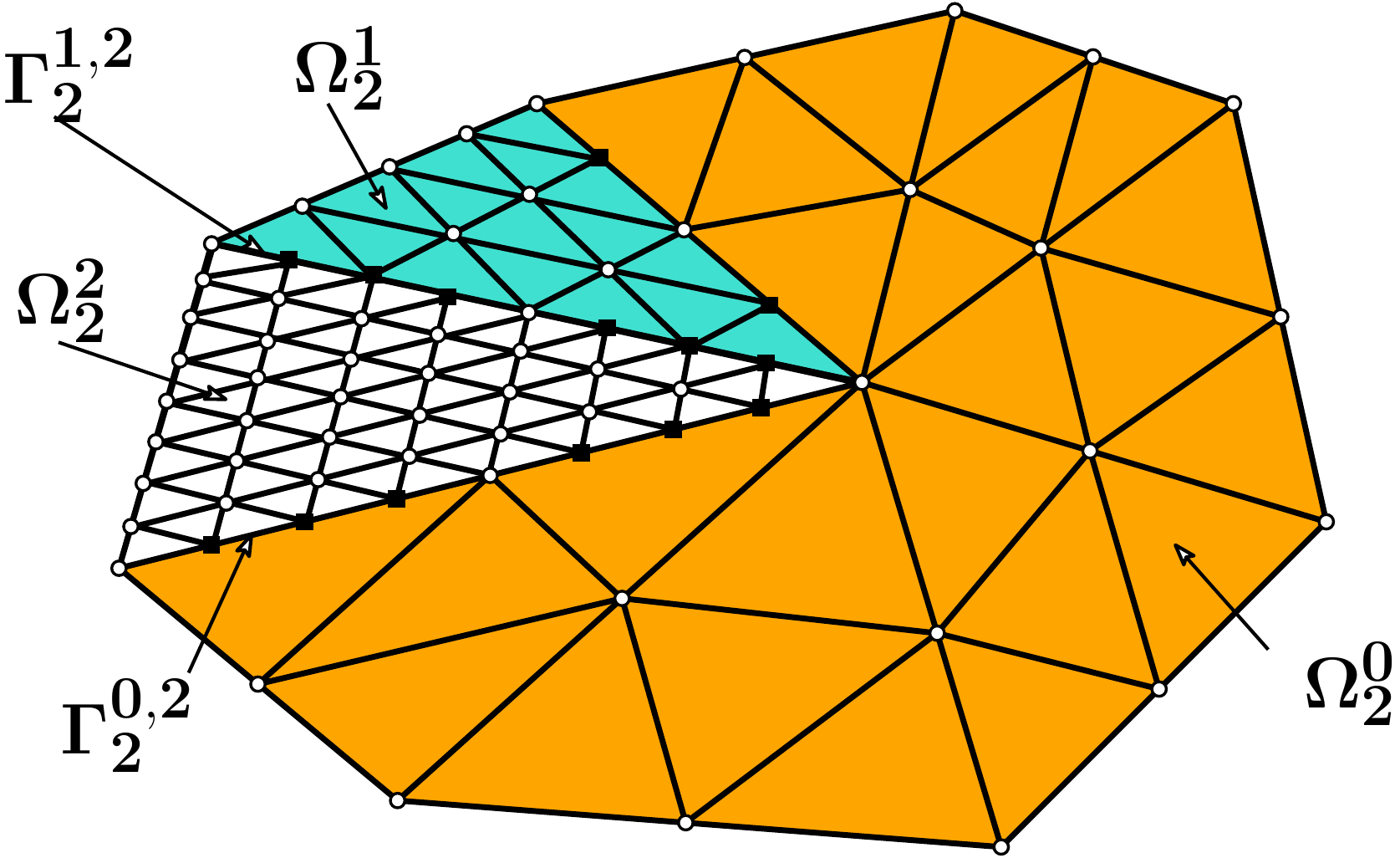}
\hspace{0.5cm}
\includegraphics[scale=0.32]{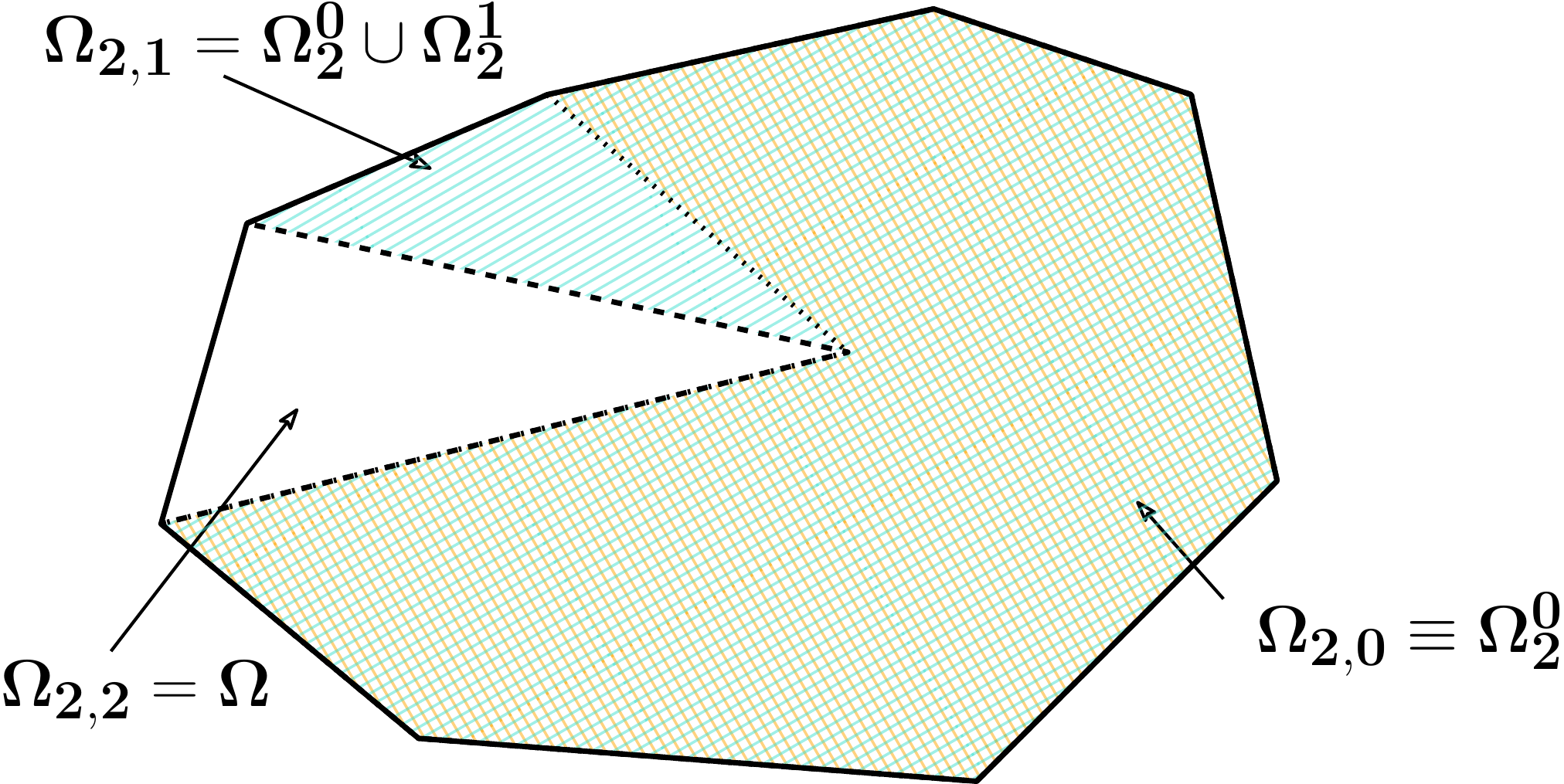}
\caption{\new{An irregular grid with hanging nodes in $2$D with $J=2$. The subdomains involved in the formulation are also highlighted}}\label{irregular_grid}
\end{figure}

Using the triangulations we defined above, for all $k=0, \dots, J$, we can define finite element spaces in the following way: 
\begin{align}\label{fem_spaces}
V^l_k &= \{ v \in H^1(\Omega^l_k) : v|_{T} \in \mathcal{P}_{\alpha}, \, \,  \forall \,T \in \mathcal{T}^l_k \},
\quad l=0, \dots, k, \\
V_k &= \{ v \in H^1(\Omega) : 
v|_{\new{int(T)}} \in \mathcal{P}_{\alpha}, \, \,  \forall \,T \in \mathcal{T}_k \} , \nonumber \\
\widehat{V}_k &= \{ \widehat v \in H^1(\Omega) \cap C^0(\Omega) : 
\widehat v|_{T} \in \mathcal{P}_{\alpha}, \, \,  \forall \,T \in \mathcal{T}_k  \} , \nonumber
\end{align}
where $\mathcal{P}_{\alpha}$ represents the set of polynomials of degree $\alpha$.
Note that the spaces $V^l_k$ are continuous by construction since the functions are defined on regular grids.
The space $V_k$ is built on an irregular triangulation and contains continuous and discontinuous functions. 
For the space $\widehat{V}_k$, we enforce continuity by removing from $V_k$ the hanging nodal functions, and enriching specific basis functions with contributions from the removed ones. 
An explicit characterization of such finite element spaces will be discussed in the next sections.

\subsection{Construction of the uniform finite element spaces $\mathbf{V_k^l}$}
Let  $X_k^l$ be the set of nodal points of the triangulation $\mathcal{T}^l_k$, 
$X_{k,l} = \cup_{m=0}^l X_k^m$ be the set of nodal points of $\mathcal{T}_{k,l}$
and $X_k = X_{k,k} $ be the set of nodal points of $\mathcal{T}_k$.
Denote with $X \equiv X_J$ the set of nodal points of the triangulation $\mathcal{T}_J$. 
Also, let $card(X_k) = N_k$ and $card(X) = N$. Depending on $\alpha$,
a nodal point could be a vertex, an edge mid-point, a face center, an element center, and more.
\begin{remark}
 There exist points $x_i \in X_k$ for which $x_i \in \bigcap\limits_{l \in A_{i,k}} X^l_k$, for some
 set $A_{i,k} \subseteq \{0,\ldots,k \}$. 
\end{remark}
This means that there are nodal points that belong to different $X^l_k$, since by Definition \ref{Omegadownup}, the $\Omega_k^l$ are allowed to overlap on portions of their boundaries.
\begin{definition}
 We define  $\varphi_{k,i}^l$ to be the standard Lagrangian nodal basis of $V^l_k$ associated with $x_i^l \in X^l_k$.
\end{definition}
Three considerations about the functions $\varphi^l_{k,i}$ are listed:
\begin{itemize}[leftmargin= * ]
\item The functions $\varphi^l_{k,i}$ satisfy the delta property
\begin{align}\label{delta_property}
\varphi_{k,i}^l(x_j^l)=\delta_{ij}\, \mbox{ for all } x_j^l\in X^l_k.
\end{align}
and they are not defined outside $\Omega_k^l$.

\item There exist basis functions $\varphi_{k,i}^m$ and $\varphi_{k,i}^n$,
defined on different subdomains $\Omega^m$ and $\Omega^n$ associated with the same nodal
point, $x_i^m = x_i^n = x_i \in \Gamma_k^{m,n}$ (the \textit{master} node). 
The two functions take different values on the shared interface $\Gamma_k^{m,n}$ and if $m<n$, then $supp(\varphi_{k,i}^n|_{\Gamma_k^{m,n}}) \subset supp(\varphi_{k,i}^m|_{\Gamma_k^{m,n}})$. 
\item Let $m<n$, there exist basis functions $\varphi_{k,i}^n$ defined on the finer triangulation associated with nodal points $x_i \in \Gamma_k^{m,n} $ (the hanging nodes)
 for which no corresponding functions $\varphi_{k,i}^m$ defined on the coarse triangulation exist.
\end{itemize}

\begin{remark}\label{disjointness}
 If for any given $m$ and $n$ the point $x_i \in X_k$ belongs to $\Gamma_k^{m,n}$, then it is either a master or a hanging node.
 Nodes in $X_k$ that are neither master nor hanging will be referred to as interior nodes.
\end{remark}

\subsection{Construction of the discontinuous finite element space $\mathbf{V_k}$}
The next goal is to build basis functions $\varphi_{k,i}$ on the whole domain $\Omega$, 
that in the interior of each subdomain $\Omega_k^l $ take the corresponding value of $\varphi^l_{k,i}$, 
and on a shared interface $\Gamma_k^{m,n}$ take the value from the coarser triangulation.
Assuming $\varphi^m_{k,i}$ is the function associated with the coarser triangulation, such a value is given by $\varphi^m_{k,i}$,
if $x_i$ is a master node, and it is zero if $x_i$ is a hanging node. 
\begin{definition}\label{Ek}
For $k \leq J$ and $l=1,\ldots,k$ we define the sets 
\begin{align}
\mathcal{E}_k^l \equiv \Omega_{k,l} \setminus \Omega_{k,l-1},
\end{align}
and $\mathcal{E}_k^0 \equiv \Omega_{k,0}$.
\end{definition}

Note that $\mathcal{E}_k^l \neq \Omega_{k}^{l}$ in general, 
since the sets $\{\Omega_{k}^{l}\}_{l=0}^k$ in Definition \ref{Omegadownup} are not disjoint. 
Any point $x \in \mathcal{E}_k^l$ is either in the interior of $\Omega^l_k$ 
or on the shared interface $\Gamma_k^{l,n}$, with $n>l$, 
but it is not on the shared interface $\Gamma_k^{l,m}$, with $m<l$.
It follows from Definition \ref{Ek} that  $\{\mathcal{E}_k^l\}_{l=0}^k$ is a pairwise disjoint collection of subsets of $\Omega$ and
\begin{align}
 \Omega = \bigcup\limits_{l=0}^k \mathcal{E}_k^l.
\end{align}
Therefore,  $\{\mathcal{E}_k^l\}_{l=0}^k$ represents a disjoint cover of $\Omega$, and so
if $x \in \Omega$, then there exists a unique $\gamma \in \{ 0,1,\ldots,k\}$ such that $x \in \mathcal{E}_k^{\gamma}$.
\begin{definition}\label{alg1}
For all $k=0, \ldots, J$ and $x_i \in X_k$, the functions $\varphi_{k,i}$ 
are defined for all $x$ in $\Omega$ in the following way:
 
 $$ \varphi_{k,i}(x) = 
 \left\{
 \begin{array}{l l}
 \varphi^{\gamma}_{k,i}(x) & \mbox{if }x_i \in X_k^{\gamma} \\ 
 0 & \mbox{otherwise}
 \end{array}
\right.,$$
where $\gamma$ is the unique integer for which $x \in \mathcal{E}_k^{\gamma}$.
\end{definition}
Let us highlight the main features of the functions $\varphi_{k,i}$: 
\begin{itemize} [leftmargin= * ]
\item The functions $\varphi_{k,i}$ can be either continuous or discontinuous and are defined on all $\Omega$.
\item $\varphi_{k,i}$ is continuous if and only if its nodal point is located in the interior 
of some $\Omega_k^l$ (where there is a regular triangulation $\mathcal T_k^l$) or on the exterior boundary if surrounded by  a regular triangulation.
When $\varphi_{k,i}$ is continuous, it can be considered the zero extension of $\varphi^l_{k,i}$ to the whole domain, 
and it satisfies the delta property
\begin{equation}
\varphi_{k,i}(x_j)=\delta_{ij}\, \mbox{ for all } x_j\in X_k \label{delta_property2}.
\end{equation}
\item $\varphi_{k,i}$ is discontinuous if and only if its nodal point is either 
a master or a hanging node.
\item A function $\varphi_{k,i}$
associated with a master node does not satisfy the delta property \eqref{delta_property2} 
for some $x_j$, with $x_j$ being a hanging node.
\item A function $\varphi_{k,i}$
associated with a hanging node does not satisfy the delta property \eqref{delta_property2},
but instead it satisfies the zero property
\begin{equation}
\varphi_{k,i}(x_j)=0 \, \mbox{ for all } x_j\in X_k. \label{zero_property}
\end{equation}
\item
Let $x_j \in X_k$, hence there exists at least one $X_k^l$ for which $x_j \in X_k^l$.
For all $x_i \in X_k$, we have
\begin{align}\label{lim_delta}
 \lim\limits_{x\in \mathcal {E}_k^{l} \rightarrow x_j} {\varphi}_{k,i}(x) = \delta_{ij}.
\end{align}
In fact, by Definition \ref{alg1}, on $\mathcal {E}_k^{l}$ 
the function ${\varphi}_{k,i}$ is either zero or equal the function $\varphi^{l}_{k,i}$ which is continuous and satisfies the delta property
\eqref{delta_property}.
\end{itemize}
We now have to introduce a more complex idea than the concept of master and hanging nodes.
\begin{definition}
A node $x_i \in X_k$ is a {\it{parent}} of a node $x_j \neq x_i$, $x_j \in X_k$ if 
$$\varphi_{k,i}(x_j) \neq 0.$$
The node $x_j$ is called a {\it{child}} of $x_i$.
\end{definition}
Parent nodes are master nodes while child nodes are hanging nodes.
If a parent node $x_i$ has a child $x_j$ which itself has a child $x_n$, then
$x_i$ is a {\it{grandparent}} of $x_n$ and $x_n$ is a  {\it{grandchild}} of $x_i$.
Note that being both a child and a parent, $x_j$ is both a master and a hanging node.
Hence $\varphi_{k,j}(x_j)=0$, but $\varphi_{k,j}(x_n) \neq 0$.
Nodes that are both master and hanging ultimately are considered hanging nodes, as the following definition states.
\begin{definition}\label{fundSets}
Let $k \in \{0,1,\ldots,J\}$ be fixed and recall that $\Omega = \bigcup\limits_{l=0}^{k} \Omega_k^l$.
We define the following sets
\new{
\begin{align*}
\textit{Interior}&:& \,\mathcal{I}_k = \{ x_i \in X_k \,\, | \,\, x_i \in int(\Omega_k^l) \,\, \mbox{for some} \,\, l\},\\
\textit{Hanging}&:& \,\mathcal{H}_k = \{ x_i \in X_k \,\, | \,\, \varphi_{k,j}(x_i) \neq 0 \,\, \mbox{for at least one} \,\, x_j \neq x_i \},\\
\textit{Master}&:& \,\mathcal{M}_k = \{ x_i \in X_k \,\, | \,\, \varphi_{k,i}(x_j) \neq 0 \,\, \mbox{for at least one} \,\, x_j \in \mathcal{H}_k, \mbox{and } x_i \notin \mathcal{H}_k \}.
\end{align*}
}\end{definition}
The set $\mathcal{I}_k$ is the set of interior nodes at the level $k$.
We refer to $\mathcal{M}_k$ as the set of master nodes at the level $k$, and it is the set of all nodes that are parents without ever being children.
We call $\mathcal{H}_k$ the set of hanging nodes at level $k$, and it is the set of all nodes that are children at least once. It includes all the parents that are also children. Note that the sets in Definition \ref{fundSets} are disjoint and that $\mathcal{I}_k \cup \mathcal{M}_k \cup \mathcal{H}_k = X_k$, recall Remark \ref{disjointness}.
\new{
\begin{lemma}\label{phi_li}
 The set $\{ \varphi_{k,i} : x_i \in X_k\} $ is linearly independent for all $k=0, \ldots, J$.
\end{lemma}
\begin{proof}
Let $\varphi_{k,j}$ be associated with $x_j \in X_k$ and assume by contradiction that 
\begin{equation} \label{contrad}
\varphi_{k,j}(x) = \sum\limits_{i: x_i \in X_k, x_i \neq x_j} c_i \varphi_{k,i}(x),
\end{equation}
where $c_i$ are some real coefficients, not all equal to zero.
Recalling Eq.~\eqref{lim_delta}, there exists at least one $X_k^l$ for which $x_j \in X_k^l$, and for all $x_i \in X_k$,
$$  \lim\limits_{x\in \mathcal {E}_k^{l} \rightarrow x_j} {\varphi}_{k,i}(x) = \delta_{ij}.$$
Hence, the following contradiction is obtained
$$ 1 = \lim\limits_{x\in \mathcal {E}_k^{l} \rightarrow x_j} \varphi_{k,j}(x) = 
\lim\limits_{x\in \mathcal {E}_k^{l} \rightarrow x_j} \sum\limits_{i: x_i \in X_k, x_i \neq x_j} c_i \varphi_{k,i}(x) = 0.
$$
\end{proof}
}
With the above lemma we can formally define the space $V_k$ first introduced in \eqref{fem_spaces}.
\begin{definition}\label{def1}
 We define $V_k$ to be the set spanned by $\{ \varphi_{k,i}\}_{i=1}^{N_k}$, namely 
 $V_k = span(\{ \varphi_{k,i}\}_{i=1}^{N_k})$.
 Because of Lemma \ref{phi_li}, it follows that $\{ \varphi_{k,i}\}_{i=1}^{N_k}$ also forms a basis for $V_k$,
 where $V_k$ is a vector space with the standard addition and scalar multiplication for real valued functions.
\end{definition}

\subsection{Construction of the continuous finite element space $\mathbf{\widehat V_k}$}
The next step is to build basis functions $\widehat\varphi_{k,i}$ on the whole domain $\Omega$ that are also continuous.
Let $x_i \in \mathcal{M}_k$ and define the set $\mathcal{H}^1_{x_i}$ as the set of all hanging nodes associated with node $x_i$.
Nodes in $\mathcal{H}^1_{x_i}$ are children of $x_i$.
With the same idea, we can define the following sets:
\begin{align*}
&\mathcal{H}^2_{x_i} = \{ x \, | \, x \in \mathcal{H}^1_{y} \,\, \mbox{for some} \,\, y \in \mathcal{H}^1_{x_i}\},\\
&\mathcal{H}^3_{x_i} = \{ x \, | \, x \in \mathcal{H}^1_{y} \,\, \mbox{for some} \,\, y \in \mathcal{H}^2_{x_i}\},\\[-1ex]
&\hspace{2cm}\vdots\\[-1ex]
&\mathcal{H}^{\alpha}_{x_i} = \{ x \, | \, x \in \mathcal{H}^1_{y} \,\, \mbox{for some} \,\, y \in \mathcal{H}^{\alpha-1}_{x_i}\}.
\end{align*}

The set $\mathcal{H}^2_{x_i}$ is the set of hanging nodes associated with nodes in $\mathcal{H}^1_{x_i}$, therefore any node in $\mathcal{H}^2_{x_i}$
is a {\it{grandchild}} of $x_i$.
The set $\mathcal{H}^3_{x_i}$ is the set of hanging nodes associated with nodes in $\mathcal{H}^2_{x_i}$ so it contains the {\it{great-grandchildren}} of $x_i$ 
and so on.
If $x \in \mathcal{H}^{\alpha}_{x_i}$ we say that $x$ is an hanging node associated with $x_i$ of {\it{degree}} at least $\alpha$.
Note that a given node can be a hanging node for more than just one master (a given node can be the child of more than one parent).
Therefore, for given $\alpha$, we introduce the sequence
$\{x_{i,\beta}^{\alpha}\}_{\beta=1}^{\beta^{max}_{\alpha}}$: this is the sequence of hanging nodes of degree $\alpha$ associated with the master 
node $x_i$ but considering {\it repetition}. Observe that $\{x_{i,\beta}^{\alpha}\}_{\beta=1}^{\beta^{max}_{\alpha}} \subseteq \mathcal{H}^{\alpha}_{x_i}$.
Moreover, we also point out that for $\alpha =1$ the elements of the sequence $\{x_{i,\beta}^{\alpha}\}_{\beta=1}^{\beta^{max}_{\alpha}}$ do not repeat and
if $\alpha=0$ there is only $x_{i,1}^{0} \equiv x_i$.
We define the sets
\begin{align*}
&\mathcal{H}_{x_i}^{1,1} \equiv \mathcal{H}_{x_i}^{1},\\
&\mathcal{H}_{x_i}^{\alpha,\beta} = \{ x \in \mathcal{H}_{x_i}^{\alpha} \, | \, x \in \mathcal{H}^1_{x_{i,\beta}^{\alpha-1}} \},\, \mbox{for} \,\, \alpha \geq 2 \,\, \mbox{and} \,\, \beta = 1, \ldots, \beta_{\alpha-1}^{max}.
\end{align*}
\begin{definition}
If $x_j \in \mathcal{H}_{x_i}^{\alpha,\beta}$ then there exist $\alpha+1$ nodes $x_{j_0},x_{j_1},x_{j_2},\ldots,x_{j_{\alpha}}$ (with $x_{j_0} \equiv x_i$ and $x_{j_{\alpha}} \equiv x_j$) 
that we refer to as the ancestors of $x_j$, such that
for all $m=0, \ldots, \alpha$, we have $x_{j_m} = x_{i,\beta_{j_m}}^m$ for some $\beta_{j_m} \in \{1,2,\ldots,\beta_m^{max}\}$.
Note that $\beta_{j_0} = 1$ always.
\end{definition}
The {\it ancestors} of $x_j$ depend on $\alpha$ and $\beta$.
However, not to make the notation too heavy, such letters do not appear
in the {\it ancestor} sequence.
Moreover, we remark that the only $m \in \{0,\ldots, \alpha \}$ for which $x_{j_{m}} \equiv x_j$ is $ m = \alpha$.

For $\alpha \geq 2$, and any $x_j \in \mathcal{H}_{x_i}^{\alpha,\beta}$, define the set
$$ U_{x_i, x_j}^{\alpha,\beta} =  \bigcup\limits_{m=1}^{\alpha-1} \mathcal{H}_{x_i}^{\alpha-m,\beta_{j_{\alpha-m-1}}}.$$
The above set represents the set of all {\it uncles} (parents included) and {\it great-uncles} (grandparents included) of $x_j \in \mathcal{H}_{x_i}^{\alpha,\beta}$.
Moreover, for $\alpha \geq 2$, we also define 
$$ S_{x_i, x_j}^{\alpha,\beta} = \{ x_{j_{\gamma}} \in \{ x_{j_m}\}_{m=2}^{\alpha} \,\, | \,\, x_{j_{\gamma}} \in U_{x_i,x_{j_{\gamma}}}^{\gamma,\beta_{j_{\gamma-1}}}\}, $$
where $x_{j_{\gamma}} \in \mathcal{H}_{x_i}^{\gamma,\beta_{j_{\gamma-1}}}$.
This is the set of all nodes in the {\it ancestor} sequence $\{ x_{j_m}\}_{m=1}^{\alpha}$ of $x_j$ that are {\it uncles} or {\it great-uncles} of {\it themselves}.
Finally, we introduce the set of all nodes in $\mathcal{H}_{x_i}^{\alpha,\beta}$ 
for which no ancestor is {\it uncle} or {\it great-uncle} of {\it itself}, namely
\begin{align*}
&\widetilde{\mathcal{H}}_{x_i}^{1,1} = \mathcal{H}_{x_i}^{1,1},\\
&\widetilde{\mathcal{H}}_{x_i}^{\alpha,\beta} = \{x_j \in \mathcal{H}_{x_i}^{\alpha,\beta} \,\, | S_{x_i, x_j}^{\alpha,\beta} = \emptyset\}, \,\, \mbox{for} \,\, \alpha \geq 2 \,\, \mbox{and} \,\, \beta = 1, \ldots, \beta_{\alpha-1}^{max}.
\end{align*}

Let $\beta_0^{max} \equiv 1$ and $\alpha_i^{max} = \max \{\alpha | \widetilde{\mathcal{H}}_{x_i}^{\alpha, \beta} \neq \emptyset, \,\, \mbox{with} \,\, \beta=1, \ldots, \beta_{\alpha-1}^{max} \}$.
\begin{definition}\label{alg2}
  For all $x_i \in X_k$, we define functions $\widehat\varphi_{k,i}$ using the previously introduced 
  functions $\{\varphi_{k,i}\}_{i=1}^{N_k}$ in the following way
$$ \widehat{\varphi}_{k,i} = 
 \left\{
 \begin{array}{l l}
 \varphi_{k,i} & \mbox{if} \,\, i \in \mathcal{I}_k \\ 
 \varphi_{k,i} + \sum\limits_{\alpha=1}^{\alpha_i^{max}} \sum\limits_{\beta=1}^{\beta^{max}_{\alpha-1}}\Big( \sum\limits_{j: x_j \in \mathcal{\widetilde{H}}^{\alpha,\beta}_{x_i} } \Big( \prod\limits_{m=1}^{\alpha} \varphi_{k,j_{m-1}}(x_{j_m})\Big) \varphi_{k,j} \Big) & \mbox{if} \,\, i \in \mathcal{M}_k \\
 0 & \mbox{if} \,\, i \in \mathcal{H}_k 
  \end{array}
\right..
  $$
\end{definition}
We reorder the basis functions just defined so that the interior nodes come first, the master nodes second and the hanging nodes last.
In matrix-vector form, the vector ${\widehat {\bm \varphi}_k}$ of the reordered $\widehat{\varphi}_{k,i}$ functions can be written as:
\begin{equation}
{\widehat {\bm \varphi}_k} = 
\left[\begin{array}{l}
 {\widehat {\bm \varphi}_{k,\mathcal{I}_k}}\\
 {\widehat {\bm \varphi}_{k,\mathcal{M}_k}}\\
 {\widehat {\bm \varphi}_{k,\mathcal{H}_k}}
\end{array}
\right]
=\left[ \begin{matrix}
I &  0 & 0\\
0 &  I & \Phi_k \\
0 &  0 & 0
\end{matrix}
\right]
\left[\begin{array}{l}
{{\bm \varphi}_{k,\mathcal{I}_k}}\\
{{\bm \varphi}_{k,\mathcal{M}_k}}\\
{{\bm \varphi}_{k,\mathcal{H}_k}}
\end{array}
\right]
=
\widehat R_k {\bm \varphi_k}. \label{matrix_vector}
\end{equation}
\begin{remark}
\new{
The hanging nodes are considered in the definition of the functions $\widehat{\varphi}_{k,i}$ to maintain the same dimension for vectors and matrices used in either $V_k$ or $\widehat V_k$.
In this way, an identity mapping between the degrees of freedom of $V_k$ and $\widehat V_k$ is preserved.
This is convenient for the implementation of the algorithm, especially in a parallel environment,
and, as it will be shown later, for extending the formulation to a multilevel setting.}
\end{remark}
Let us now consider two simple examples that do not fall in the case of 1-irregular meshes. 
These are illustrative to better understand the algorithm described in this section. Their generalization
to complex arbitrary-level hanging node meshes as the ones used in the next sections is straightforward.

We start with a two-dimensional example to visualize the construction of a function associated with a node
in $\mathcal{M}_k$, as described in Definition \ref{alg2}.
The triangulation considered in the example is depicted in Figure \ref{exampleTri}. This can be part of a larger triangulation
but for simplicity we just focus on the portion in the figure.
We consider bilinear triangular elements with a maximum of two level local refinement (so $J=2$) and 
focus on the master node $x_1 \in \mathcal{M}_2$.

\begin{figure}[!b]
\begin{center}
\begin{minipage}[t]{0.5\linewidth}
\centering
\vspace{-1.8cm}\includegraphics[scale=0.45]{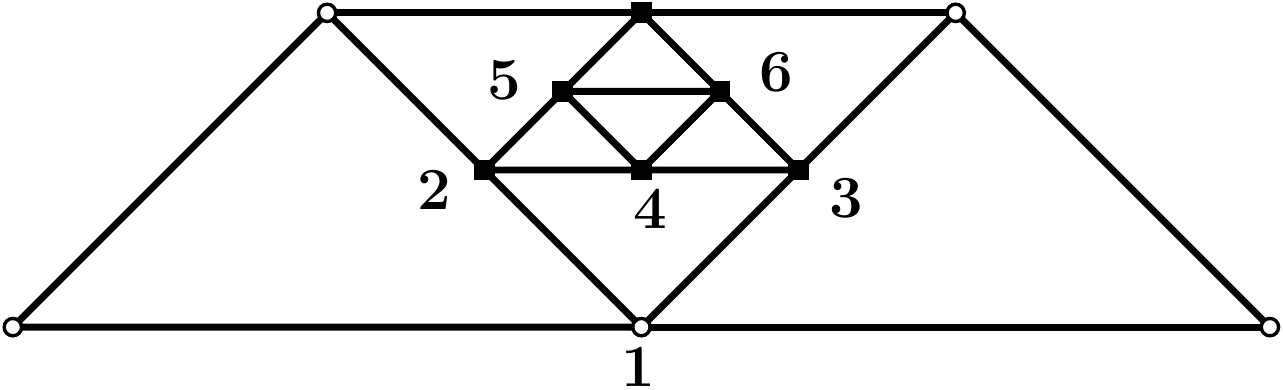}
\caption{Example of a two-dimensional bilinear grid with a two level local refinement.}\label{exampleTri}
\end{minipage} 
\hspace{0.125in}
\begin{minipage}[t]{0.4\linewidth}
\centering
\includegraphics[scale=0.5]{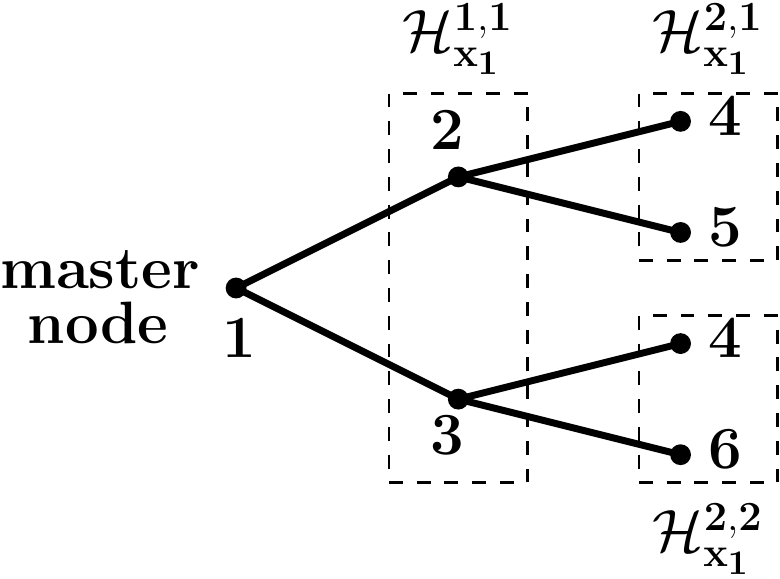}
\caption{Tree diagram that displays the family tree of node $x_1$.}\label{treeTri}
\end{minipage}
\end{center}
\end{figure}

\begin{figure}[!b]
\begin{center}
\begin{minipage}[t]{0.52\linewidth}
\centering
\includegraphics[scale=0.4]{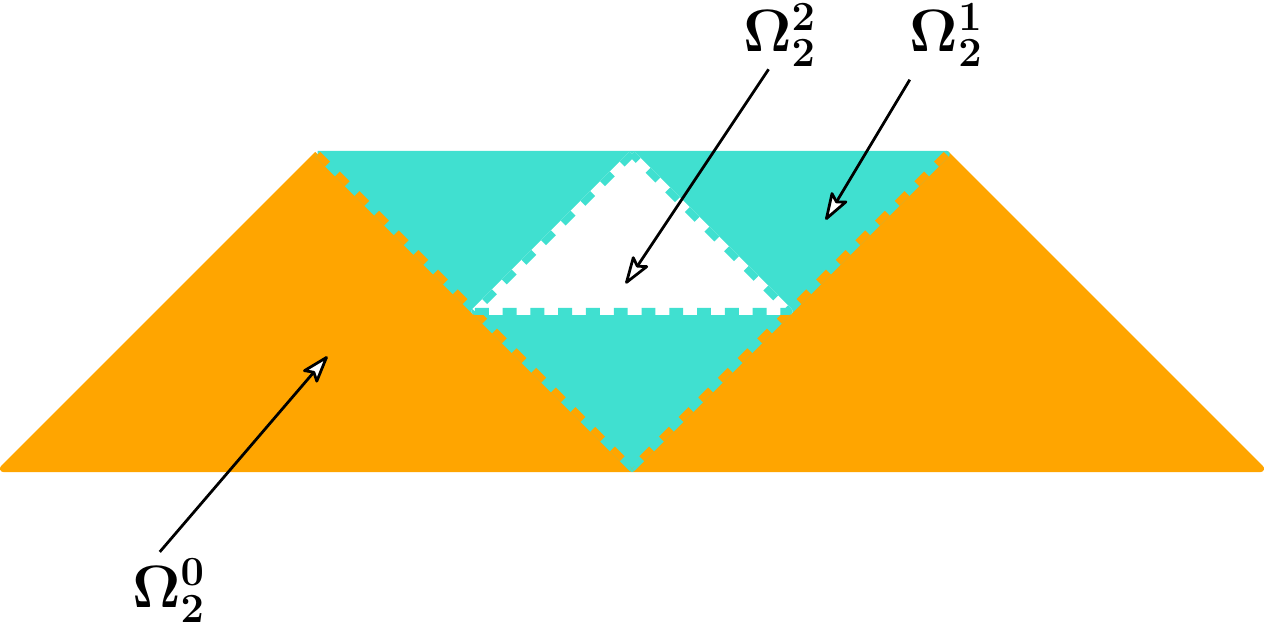}
\caption{Subdomains associated with the different degrees of refinement of the grid in Figure \ref{exampleTri}.}\label{exampleTriOmegas}
\end{minipage} 
\hspace{0.125in}
\begin{minipage}[t]{0.4\linewidth}
\centering
\includegraphics[scale=0.4]{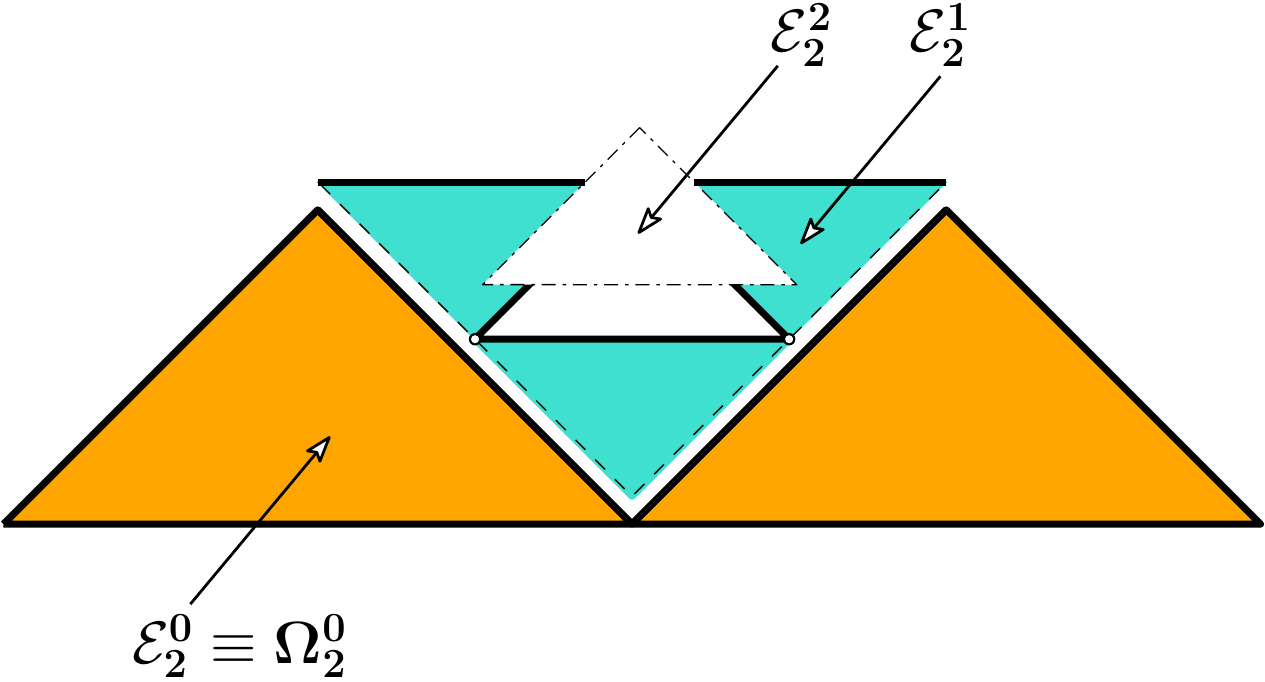}
\vspace{-0.4cm}\caption{Collection $\{\mathcal{E}_2^l\}_{l=0}^{2}$ associated with the grid in Figure \ref{exampleTri}.}\label{exampleTriE}
\end{minipage}
\end{center}
\end{figure}


Let us start describing the sets and sequences involved in Definition \ref{alg2} for this specific example.
\begin{align*}
& \mathcal{H}_{x_1}^1 = \{x_2,x_3\}, \quad \mathcal{H}_{x_1}^2 = \{x_4, x_5, x_6\}, \\
& \{x_{1,\beta}^{1}\}_{\beta=1}^{2} = \{ x_2, \, x_3\},\qquad
 \{x_{1,\beta}^{2}\}_{\beta=1}^{4} = \{ x_4, \, x_5, \, x_4, \,x_6\},\\
&\widetilde{\mathcal{H}}_{x_1}^{1,1} = \mathcal{H}_{x_1}^{1,1} =\{ x_2, x_3\},\quad
\widetilde{\mathcal{H}}_{x_1}^{2,1} = \mathcal{H}_{x_1}^{2,1} = \{x_4, x_5\}, \quad \widetilde{\mathcal{H}}_{x_1}^{2,2} = \mathcal{H}_{x_1}^{2,2} = \{x_4, x_6 \}.
\end{align*} 
Note that, in this case, there are no hanging nodes that are uncles or great-uncles of themselves.
In Figure \ref{treeTri}, the family tree associated with node $x_1$ is reported. \new{ 
\begin{remark}
In the construction of the tree structure of the mesh, the child of a master node is
always located on the same edge (or face) of the parent. Only when considering grandchildren,
great-grandchildren and so on, the search extends to neighboring nodes. This situation is embedded in
the nature of arbitrary-level hanging nodes, and it cannot be avoided \cite{vsolin2008arbitrary,vsolin2004goal}.
For such a reason, an exclusive edge-by-edge approach is not possible in case
of arbitrary-level hanging nodes.
The search for the children, however, remains localized on the support of the coarse
master node basis, so the proposed approach is still local.
\end{remark}
}
Using Definition \ref{alg2} we obtain 
\begin{align*}
  \widehat\varphi_{k,1} &= \varphi_{k,1} + \varphi_{k,1}(x_2) \varphi_{k,2} + \varphi_{k,1}(x_3) \varphi_{k,3} + \varphi_{k,1}(x_2)\varphi_{k,2}(x_4)\varphi_{k,4} \\
 & + \varphi_{k,1}(x_2)\varphi_{k,2}(x_5)\varphi_{k,5} + \varphi_{k,1}(x_3)\varphi_{k,3}(x_4)\varphi_{k,4} +\varphi_{k,1}(x_3)\varphi_{k,3}(x_6)\varphi_{k,6} \\
 & = \varphi_{k,1} + 0.5 \, \varphi_{k,2} + 0.5 \, \varphi_{k,3} + 0.5  \cdot  0.5 \, \varphi_{k,4} + 0.5 \cdot 0.5 \, \varphi_{k,5}  \\
 & + 0.5 \cdot 0.5 \, \varphi_{k,4} + 0.5 \cdot 0.5 \, \varphi_{k,6} \\
 & = \varphi_{k,1} + 0.5 \, \varphi_{k,2} + 0.5 \, \varphi_{k,3} + 0.25\, \varphi_{k,4} + 0.25 \, \varphi_{k,5} + 0.25\varphi_{k,4} + 0.25\varphi_{k,6} \\
 & = \varphi_{k,1} + 0.5 \, \varphi_{k,2} + 0.5 \, \varphi_{k,3} + 0.5\, \varphi_{k,4} + 0.25 \, \varphi_{k,5} + 0.25\varphi_{k,6},\\
 \widehat\varphi_{k,2} &\equiv 0,\quad
 \widehat\varphi_{k,3} \equiv 0, \quad
 \widehat\varphi_{k,4} \equiv 0, \quad
 \widehat\varphi_{k,5} \equiv 0,  \quad
 \widehat\varphi_{k,6} \equiv 0. 
\end{align*}
The matrix-vector form corresponding to Eq.~\eqref{matrix_vector} is given by:
\begin{equation*}
\left[\begin{array}{l}
 \widehat\varphi_{k,1}\\
 \widehat\varphi_{k,2}\\
 \widehat\varphi_{k,3}\\
 \widehat\varphi_{k,4}\\
 \widehat\varphi_{k,5} \\
 \widehat\varphi_{k,6}
\end{array}
\right]
=\left[ \begin{matrix}
1 & 0.5 & 0.5   &0.5 & 0.25 &0.25\\
0 & 0 & 0 & 0 & 0 &0\\
0 & 0 & 0   &0 &0 &0\\
0 & 0 & 0 &0 &0 &0\\
0 & 0 & 0 &0 &0 &0 \\
0 & 0 & 0 &0 &0 &0
\end{matrix}
\right]
\left[\begin{array}{l}
 \varphi_{k,1}\\
 \varphi_{k,2}\\
 \varphi_{k,3}\\
 \varphi_{k,4}\\
 \varphi_{k,5} \\
 \varphi_{k,6}
\end{array}\right]. 
\end{equation*}

We continue with a three-dimensional example of an L-shaped domain, visible in Figure \ref{example3D}.
Once again the maximum degree of local refinement is two, so $J=2$, and we only focus on the node $x_1 \in \mathcal{M}_2$.
We present this specific example because the sets $\mathcal{H}_{x_1}^{2,1}$
and $\widetilde{\mathcal{H}}_{x_1}^{2,1}$ are different in this case and so
it will be clear why we consider the sets $\widetilde{\mathcal{H}}_{x_1}^{2,1}$ instead of $\mathcal{H}_{x_1}^{2,1}$
in Definition \ref{alg2}.
In Figures \ref{example3DOmegas} and \ref{example3DE} the subdomains associated with the different
degrees of refinement are shown together with the collection $\{\mathcal{E}_2^l\}_{l=0}^{2}$.
Since the goal of this example is to illustrate a situation where 
the sets $\mathcal{H}_{x_1}^{2,1}$ and $\widetilde{\mathcal{H}}_{x_1}^{2,1}$ are different,
we only consider the contributions given by the numbered nodes in Figure \ref{example3D}.

\begin{figure}[!t]
\begin{center}
\begin{minipage}[t]{0.5\linewidth}
\centering
\includegraphics[scale=0.4]{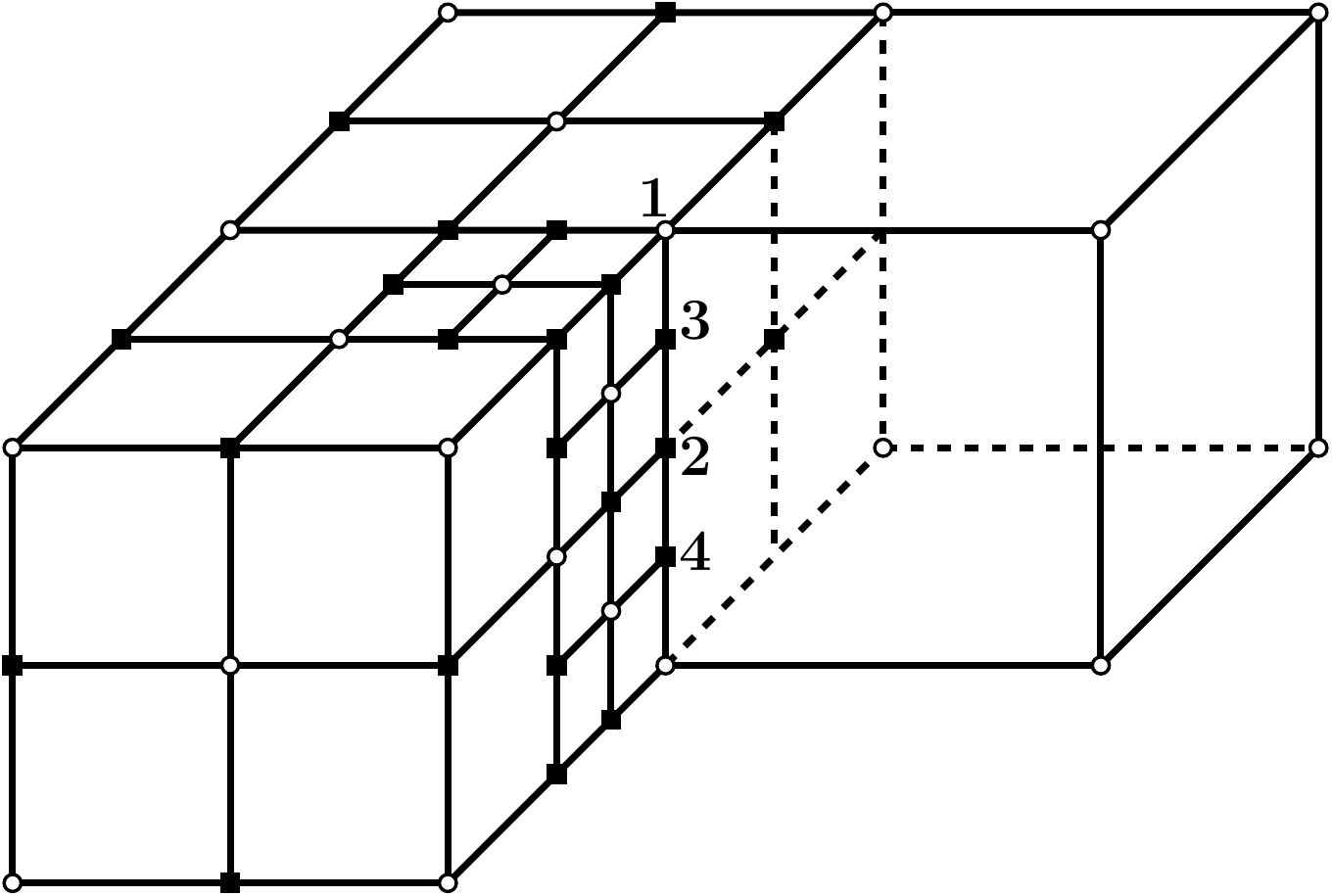}
\caption{Example of three-dimensional bilinear grid on an L-shaped domain with a three level local refinement.}\label{example3D}
\end{minipage} 
\hspace{0.125in}
\begin{minipage}[t]{0.4\linewidth}
\centering
\includegraphics[scale=0.5]{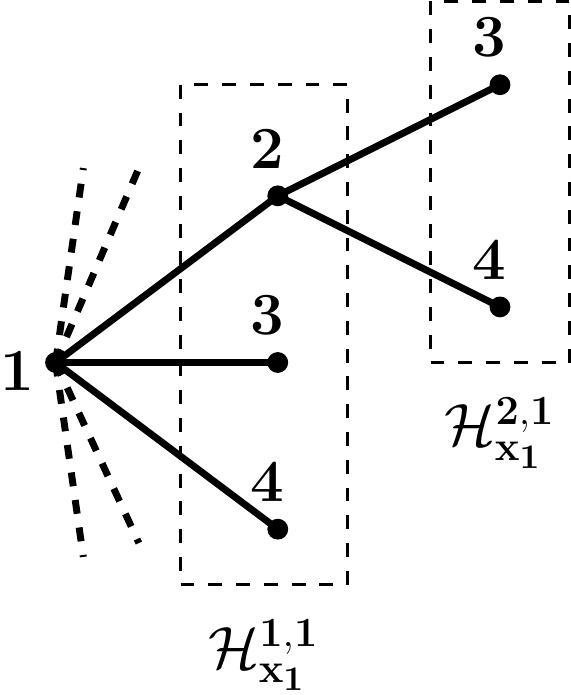}
\caption{Tree diagram that displays the family tree of node $x_1$.}\label{example3Dtree}
\end{minipage}
\end{center}
\end{figure}
\begin{figure}[!b]
\begin{center}
\begin{minipage}[t]{0.52\linewidth}
\centering
\includegraphics[scale=0.3]{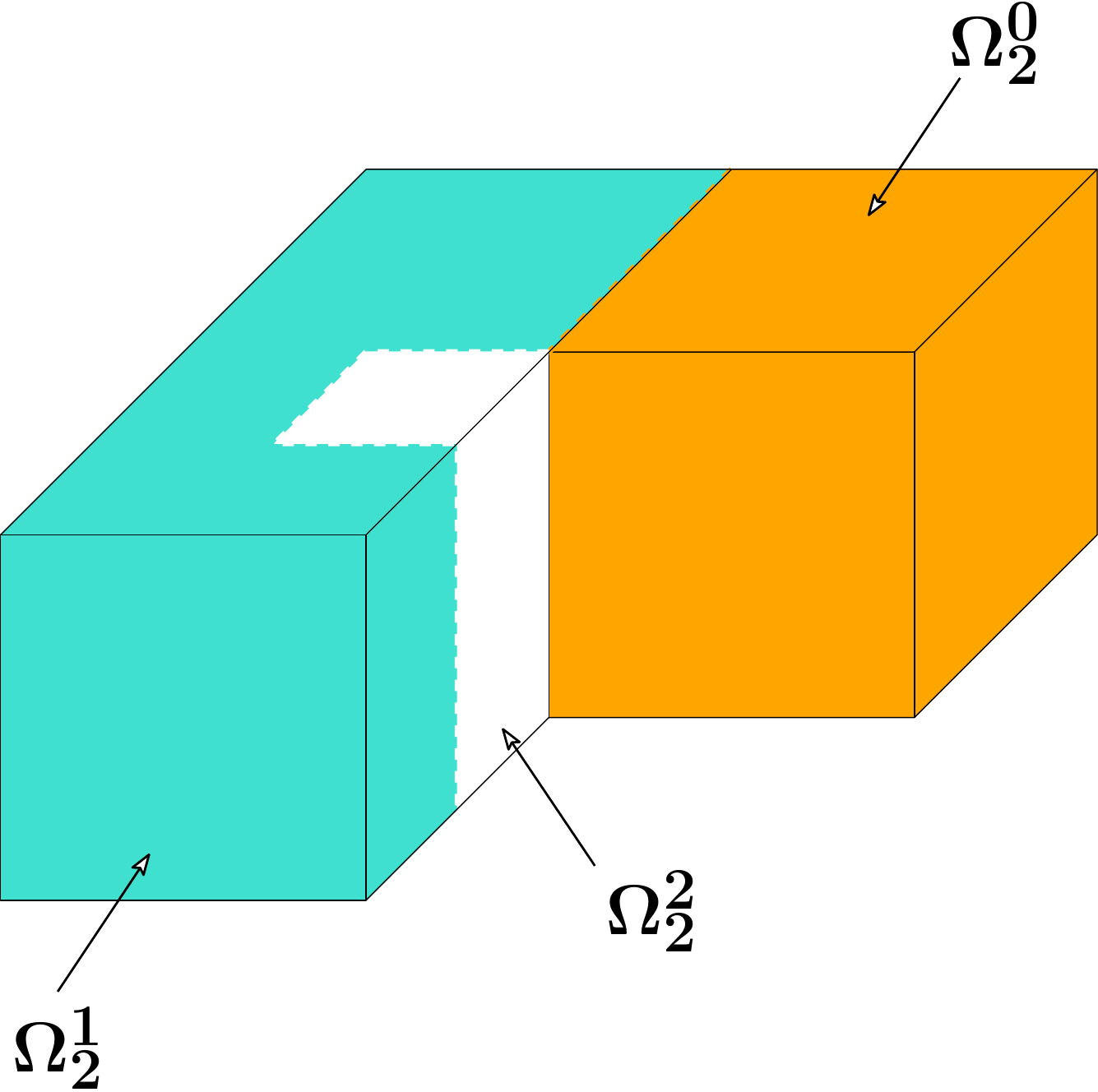}
\vspace{-0.25cm}\caption{Subdomains associated with the different degrees of refinement of the grid in Figure \ref{example3D}.}\label{example3DOmegas}
\end{minipage} 
\hspace{0.125in}
\begin{minipage}[t]{0.4\linewidth}
\centering
\vspace{-3.4cm} \includegraphics[scale=0.3]{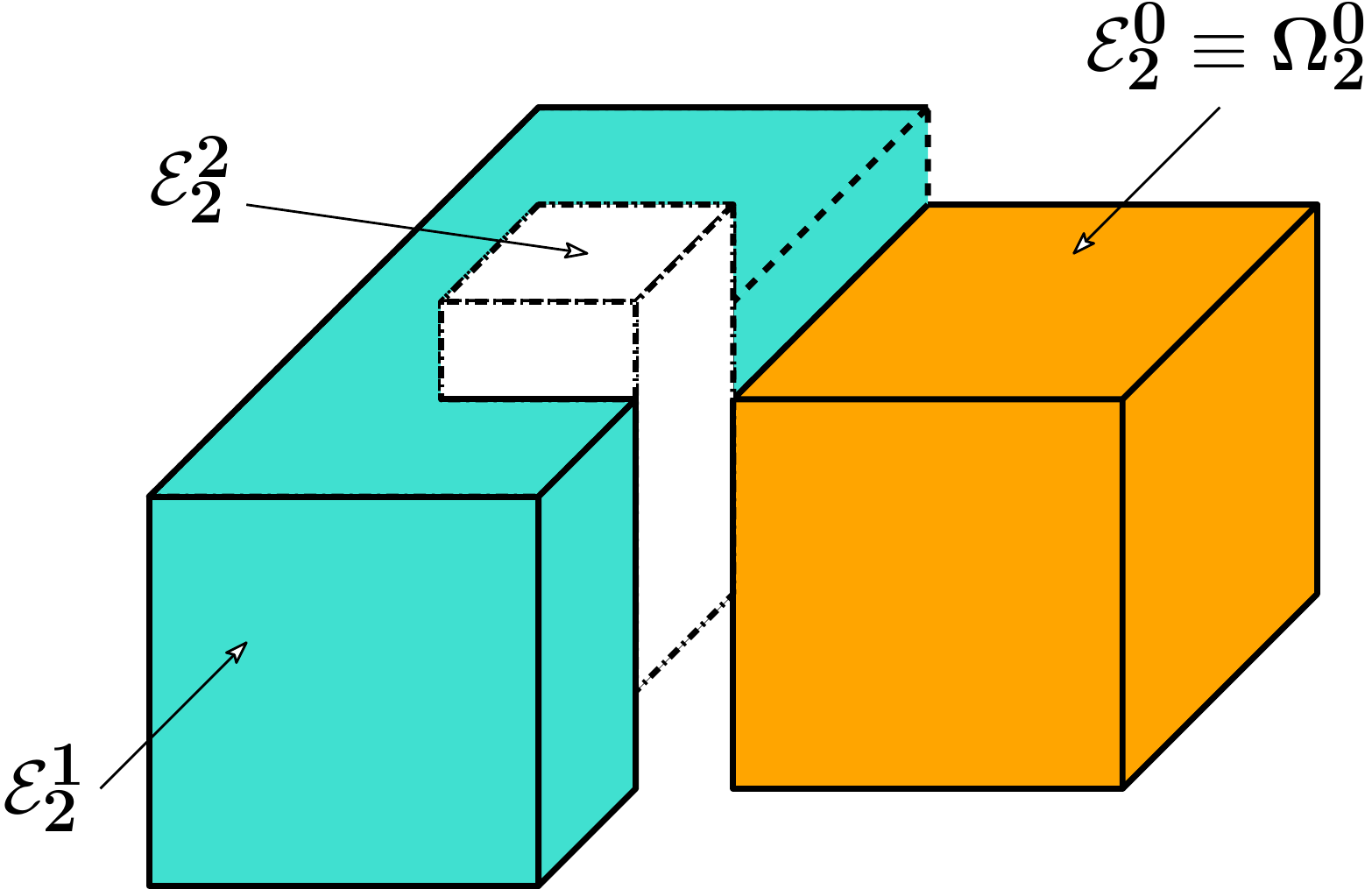}
\caption{Collection $\{\mathcal{E}_2^l\}_{l=0}^{2}$ associated with the grid in Figure \ref{example3D}.}\label{example3DE}
\end{minipage}
\end{center}
\end{figure}

Let \new{us} describe the sets and sequences involved in Definition \ref{alg2} for this three-dimensional example.
\begin{align*}
 & \mathcal{H}_{x_1}^1 = \{x_2,x_3, x_4\}, \quad \mathcal{H}_{x_1}^2 = \{x_3, x_4\},\\
 & \{x_{1,\beta}^{1}\}_{\beta=1}^{3} = \{ x_2, \, x_3, \, x_4\},\quad \{x_{1,\beta}^{2}\}_{\beta=1}^{2} = \{ x_3, \, x_4\},\\
 & \widetilde{\mathcal{H}}_{x_1}^{1,1} = \mathcal{H}_{x_1}^{1,1} =\{ x_2, x_3, x_4\}, \quad \widetilde{\mathcal{H}}_{x_1}^{2,1} = \emptyset, \quad \mathcal{H}_{x_1}^{2,1} = \{ x_3, \, x_4\}. 
\end{align*}
The reason why $\widetilde{\mathcal{H}}_{x_1}^{2,1} = \emptyset$ is because the nodes $x_3$ and $x_4$ are uncles of themselves, so they
are removed from $\mathcal{H}_{x_1}^{2,1}$. However, since this set is composed only of such nodes, it becomes the empty set.
The family tree of node $x_1$ is visible in Figure \ref{example3Dtree}. 
Using Definition \ref{alg2} we obtain
\begin{align*}
  \widehat\varphi_{k,1} &= \varphi_{k,1} + \varphi_{k,1}(x_2) \varphi_{k,2} + \varphi_{k,1}(x_3) \varphi_{k,3} + \varphi_{k,1}(x_4) \varphi_{k,4} + f_{k,1} \\
                        &= \varphi_{k,1} + 0.5 \, \varphi_{k,2} + 0.75 \, \varphi_{k,3} + 0.25 \, \varphi_{k,4} + f_{k,1},\\
 \widehat\varphi_{k,2} &\equiv 0,\quad  \widehat\varphi_{k,3} \equiv 0, \quad \widehat\varphi_{k,4} \equiv 0,
\end{align*}
where $f_{k,i}$ is a function that gives the contributions of the other hanging nodes associated with $x_1$ that have not been
made explicit in this example.
\begin{remark}
In these examples we only addressed bilinear and trilinear elements for simplicity. The same method applies
to all Lagrangian elements of any polynomial degree satisfying the delta properties. 
In the numerical example section we used bi/trilinear, quadratic, bi/triquadratic elements, and combinations of them.
\end{remark}
\begin{theorem}
The functions $ \widehat \varphi_{k,i} $ are continuous for all $x_i \in X_k$.
\end{theorem}
\begin{proof}
For any node $x_i \in \mathcal I_k$, the continuity is inherited from $\varphi_{k,i}$, while
if $x_i \in \mathcal H_k$, then $\widehat \varphi_{k,i}$ is identically $0$, thus it is continuous.
For $x_i \in \mathcal M_k$, the proof is more complex. Vaguely speaking we want to show
that if $\varphi_{k,i}$ is a discontinuous function associated with a master node $x_i$,
it is possible to build a continuous function $\widehat \varphi_{k,i}$ with the
contributions from hanging node functions,  $\varphi_{k,j}$, for some $j$.
Therefore, let $x_i \in \mathcal M_k$, we will show that $\widehat \varphi_{k,i}$ is continuous
at all hanging nodes $x_j$ associated with $x_i$. 
Note that as we have already discussed,
$x_j$ could belong to more than just one $\widetilde{\mathcal{H}}_{x_i}^{\alpha,\beta}$, meaning
that $x_j$ could be a child of more than just one parent.
Let $e_j$ be the edge (or face) containing $x_j$ that belongs to the coarsest
triangulation that has $x_j$ among its nodal points. 
On $e_j$ then, there will be  $\eta_{max}$ parent nodes $x_{j_{\eta}} \in \mathcal{H}^{\alpha}_{x_i}$ for some $\alpha$,
for which $\varphi_{j_{\eta}}(x_j) \neq 0$.
Considering that $x_j$ can posses multiple degrees as hanging node of $x_i$, 
let $\alpha_{\eta}$ be the degree of $x_j$ as child of $x_{j_{\eta}}$. Then there will be $\eta_{max}$ {\it ancestor} sequences 
$\{ x_{j_{\eta},\zeta}\}_{\zeta=0}^{\alpha_{\eta}}$, with $x_{j_{\eta},{\alpha_{\eta}}} \equiv x_j$, 
$x_{j_{\eta},\alpha_{\eta}-1} \equiv x_{j_{\eta}}$, and $x_{j_{\eta},0} \equiv x_i$.
If we denote by $\varphi_{k,j_{\eta},\zeta}$ the functions whose nodal point is $x_{j_{\eta},\zeta}$,
then $\widehat \varphi_{k,i}$ can be rewritten as
\begin{align} \label{simplified}
\widehat \varphi_{k,i} = \sum\limits_{\eta=1}^{\eta_{max}} \Big[ \Big(\prod\limits_{\zeta=1}^{\alpha_{\eta}-1}\varphi_{k,j_{\eta},\zeta-1}(x_{j_{\eta},\zeta})\Big)\varphi_{k,j_{\eta}}
+ \Big(\prod\limits_{\zeta=1}^{\alpha_{\eta}} \varphi_{k,j_{\eta},\zeta-1}(x_{j_{\eta},\zeta})\Big)\varphi_{k,j}\Big] + f_j,
\end{align}
where $f_j$ is a function for which $f_j(x_j)=0$ and $\lim\limits_{x\rightarrow x_j} f_j(x) = 0$.
Since $x_j$ is a hanging node, we have that $\varphi_{k,j}(x_j)=0$. Hence, recalling that $x_{j_{\eta},{\alpha_{\eta}}} \equiv x_j$, 
and $\varphi_{k,j_{\eta},\alpha-1} \equiv \varphi_{k,j_\eta}$, we have
$$\widehat \varphi_{k,i}(x_j) = \sum\limits_{\eta=1}^{\eta_{max}} \Big(\prod\limits_{\zeta=1}^{\alpha_{\eta}}\varphi_{k,j_{\eta},\zeta-1}(x_{j_\eta,\zeta})\Big).$$
Now, let's define the set 
\begin{align}\label{Z}
Z_j = \bigcup\limits_{\gamma}\mathcal{E}_k^{\gamma},
\end{align}
where each $\mathcal{E}_k^{\gamma}$ in the above union is such that $x_j$ belongs to the closure of $\mathcal{E}_k^{\gamma}$.
We observe that $\mathcal{E}_k^{\gamma}$ (and so $Z_j$) is independent of $\eta$ and that $e_j \subseteq \mathcal{E}_k^{\gamma}$ for some $\gamma$.
Moreover, $\lim\limits_{x \in \Omega \rightarrow x_j}\widehat \varphi_{k,i}(x) = \lim\limits_{x \in Z_j\rightarrow x_j}\widehat \varphi_{k,i}(x)$,
therefore to show existence of the limit $\lim\limits_{x \in \Omega \rightarrow x_j}\widehat \varphi_{k,i}(x)$ we have to show that 
the limits $\lim\limits_{x \in \mathcal{E}_k^{\gamma}\rightarrow x_j}\widehat \varphi_{k,i}(x)$ all exist and are equal for all $\gamma$.
On every set $\mathcal{E}_k^{\gamma}$, by definition, exactly one of the following two cases can happen:
\begin{enumerate} [leftmargin= * ]
 \item $x_{j} \in X_k^{\gamma}$, $x_j \notin \mathcal{E}_k^{\gamma}$ and $\lim\limits_{x \in \mathcal{E}_k^{\gamma}\rightarrow x_j}\varphi_{k,j_{\eta}}(x) = 0$, for all $\eta$ .
 
 The above limit is obtained for the following reason:
 if $x_{j_{\eta}}\in X_k^{\gamma}$, then $\varphi_{k,j_{\eta}} \equiv \varphi^{\gamma}_{k,j_{\eta}}$ on
 $\mathcal{E}_k^{\gamma}$, for all $\eta$, and it is true that $\lim\limits_{x \in \mathcal{E}_k^{\gamma}\rightarrow x_j}\varphi^{\gamma}_{k,j_{\eta}}(x) = 0$ .
 If $x_{j_{\eta}}\notin X_k^{\gamma}$, then $\varphi_{k,j_{\eta}} \equiv 0$ on $\mathcal{E}_k^{\gamma}$, so the same result is verified.
 Note that since $x_j \in X_k^{\gamma}$, from Definition \ref{alg1} we have that $\varphi_{k,j} \equiv \varphi^{\gamma}_{k,j}$ on $\mathcal{E}_k^{\gamma}$.
 
 \item $x_{j} \notin X_k^{\gamma}$, $x_j \in \mathcal{E}_k^{\gamma}$  and $\lim\limits_{x \in \mathcal{E}_k^{\gamma}\rightarrow x_j}\varphi_{k,j_\eta}(x) = \varphi_{k,j_\eta}(x_j)$, for all $\eta$ .
  
 In this case, $x_{j_{\eta}}\in X_k^{\gamma}$ always, and we have $\varphi_{k,j_{\eta}} \equiv \varphi^{\gamma}_{k,j_{\eta}}$ on
 $\mathcal{E}_k^{\gamma}$, for all $\eta$.
 Nodes $x_{j_{\eta}}\notin X_k^{\gamma}$ on $e_j$ cannot exist. If they did, then $x_j$ would be {\it uncle} or {\it great-uncle} 
 of itself but such nodes have been removed from the sets $\mathcal{H}_{x_i}^{\alpha,\beta}$ so this cannot happen.
  Moreover, since $x_{j} \notin X_k^{\gamma}$, and $x_j \in \mathcal{E}_k^{\gamma}$, Definition \ref{alg1} gives $\varphi_{k,j} \equiv 0$ on $\mathcal{E}_k^{\gamma}$. 
 \end{enumerate}
 
 Hence, let $\mathcal{E}_k^{\gamma}$ be given and assume we are in Case $1$. Then 
$$\lim\limits_{x \in \mathcal{E}_k^{\gamma}\rightarrow x_j}\varphi_{k,j_{\eta}}(x) = 0 \,\, \mbox{for all} \,\, \eta, \qquad \lim\limits_{x \in \mathcal{E}_k^{\gamma}\rightarrow x_j}\varphi_{k,j}(x)= \lim\limits_{x \in \mathcal{E}_k^{\gamma}\rightarrow x_j}\varphi^{\gamma}_{k,j}(x) = \delta_{jj}=1 .$$
Referring to Eq. \eqref{simplified} this implies that 
$$ \lim\limits_{x \in \mathcal{E}_k^{\gamma}\rightarrow x_j} \widehat{\varphi}_{k,i}(x) =  \sum\limits_{\eta=1}^{\eta_{max}} \Big(\prod\limits_{\zeta=1}^{\alpha_{\eta}}\varphi_{k,j_\eta,\zeta-1}(x_{j_\eta,\zeta})\Big).$$
In Case $2$, we have for all $\eta$,
$$\lim\limits_{x \in \mathcal{E}_k^{\gamma}\rightarrow x_j}\varphi_{k,j_\eta}(x) = \lim\limits_{x \in \mathcal{E}_k^{\gamma}\rightarrow x_j}\varphi^{\gamma}_{k,j_\eta}(x) = \varphi^{\gamma}_{k,j_{\eta}}(x_j) =  \varphi_{k,j_\eta}(x_j), \quad \lim\limits_{x \in \mathcal{E}_k^{\gamma}\rightarrow x_j}\varphi_{k,j}(x)= 0.$$
Again referring to Eq. \eqref{simplified} this gives 
$$ \lim\limits_{x \in \mathcal{E}_k^{\gamma}\rightarrow x_j} \widehat{\varphi}_{k,i}(x) =  \sum\limits_{\eta=1}^{\eta_{max}} \Big(\prod\limits_{\zeta=1}^{\alpha_{\eta}}\varphi_{k,j_\eta,\zeta-1}(x_{j_\eta,\zeta})\Big).$$
This proves that $\lim\limits_{x \in \Omega \rightarrow x_j} \widehat{\varphi}_{k,i}(x)$ exists and
$$\widehat{\varphi}_{k,i}(x_j) =  \sum\limits_{\eta=1}^{\eta_{max}} \Big(\prod\limits_{\zeta=1}^{\alpha_{\eta}}\varphi_{k,j_\eta,\zeta-1}(x_{j_\eta,\zeta})\Big) = \lim_{x \in \Omega \rightarrow x_j} \widehat{\varphi}_{k,i}(x).$$
Hence, the continuity of $\widehat{\varphi}_{k,i}$ at any hanging node $x_j$ of $x_i$ is proved.

To complete the proof, we have to show that $\widehat{\varphi}_{k,i}$ is continuous on all edges or faces in the triangulation $\mathcal T_k$
that contain at least one hanging node.
Continuity at any other point in the domain is inherited from the continuity of $\varphi_{k,i}$.
Let $e_{j}$ be such edge or face that contains $p$ nodes $x_{j_\ell}$, with $\ell=1,\ldots,p$,
where at least one of them is a hanging node.
Then there exist at least two elements $\Omega_{c} \in \mathcal{T}^{c}_k$ and $\Omega_{f} \in \mathcal{T}^{f}_k$ such that:
\begin{itemize} [leftmargin= * ]
 \item $e_j = \Omega_c \cap \Omega_f$,
 \item $\Omega_c \subseteq \mathcal{E}_k^{c}$ and $\Omega_f \subseteq \overline{\mathcal{E}_k^{f}}$,
\end{itemize}
where $\overline{\mathcal{E}_k^{f}}$ denotes the closure of $\mathcal{E}_k^{f}$.
We remark that the superscripts $c$ and $f$ stand for coarse and fine, respectively.
If there exist more than two elements that share the same $e_j$ (this is possible if $e_j$ is an edge in a 3D triangulation) then
we consider them in pairs, where $f$ is fixed and always refers to the element on the finest triangulation
while $c$ spans the remaining elements, one at the time.

Let us define the following function
    $$ h(x) := 
 \left\{
 \begin{array}{c l l}
 \widehat{\varphi}_{k,i}(x)& & \mbox{if} \,\, x \in int(\Omega_f) \\ 
  \lim\limits_{t \in int(\Omega_f) \rightarrow x} \widehat{\varphi}_{k,i}(t) && \mbox{if} \,\, x \in e_j 
  \end{array}
\right..
  $$
Then $h$ is a polynomial function on its domain. Let $h|_{e_j}$ be the trace of $h$ on $e_j$. Its 
value on $e_j$ is uniquely determined by the values of $h$ at the $p$ interface nodes.
Namely, the space of the traces, referred as $V_k^f\big|_{e_j} \subseteq H^{1/2}(e_j)$, has dimension $p$
and 
$$ h|_{e_j}(x) = \sum_{\ell=1}^{p} h(x_{j_\ell}) \,\, \varphi_{k,j_\ell}^f\big|_{e_j}(x).$$
Let $g(x) := \widehat{\varphi}_{k,i}(x)$ for all $x \in \Omega_c$.
The function $g$ is also a polynomial function on its domain.
Let $g|_{e_j}$  be the trace of $g$ on $e_j$ and $V_k^c\big|_{e_j}$ the corresponding space of the traces.
Since $g|_{e_j} \in V_k^c\big|_{e_j} \subseteq V_k^f\big|_{e_j}$, the function $g|_{e_j}$ can also be represented
with the bases of $V_k^f\big|_{e_j}$, namely
$$ g|_{e_j}(x) = \sum_{\ell=1}^{p} g(x_{j_\ell}) \,\, \varphi_{k,j_\ell}^f\big|_{e_j}(x).$$
Recalling that by continuity of $\widehat{\varphi}_{k,i}$ at the $p$ nodes, on $e_j$ we also have
$$h(x_{j_\ell}) = \widehat{\varphi}_{k,i}(x_{j_\ell})= g(x_{j_\ell}), \,\, \mbox{for all} \,\, \ell=1,\ldots, p,$$
and consequently $h(x) = g(x)$ for all $x \in e_j$.
Therefore, for all $x \in e_j$
$$\lim\limits_{t \in int(\Omega_f) \rightarrow x} \widehat{\varphi}_{k,i}(t) =
h(x) = g(x) = \widehat{\varphi}_{k,i}(x) =  \lim\limits_{t \in int(\Omega_c) \rightarrow x} \widehat{\varphi}_{k,i}(t).$$
Thus, $\widehat{\varphi}_{k,i}$ is continuous at every face or edge $e_j$ that contains at least one hanging node.
This completes the proof.
 \end{proof}

\begin{lemma}\label{hatphi_li}
The set $\{ \widehat \varphi_{k,i}$ : $x_i \in \mathcal I_k \cup \mathcal M_k\}$, is linearly independent.
\end{lemma}
\begin{proof}$\,$\\
If $x_i \in \mathcal I_k$ then $\widehat \varphi_{k,i} = \varphi_{k,i}$ 
and the statement follows from \eqref{delta_property2}.\\
If $x_i \in \mathcal M_k$ then, according to \eqref{matrix_vector}, each $\widehat \varphi_{k,i}$ 
is a linear combination of one master node function $\varphi_{k,i}$, 
which satisfies the delta property for all $x_j \notin \mathcal H_k$, 
and several hanging node functions $\varphi_{k,j}$, which satisfy the zero property \eqref{zero_property}. 
Hence, letting $x_j\notin \mathcal H_k$, we have
\begin{equation}
\widehat {\bm \varphi}_{k,\mathcal M}(x_j) = {\bm \varphi}_{k,\mathcal M}(x_j) + \Phi_k {\bm \varphi}_{k,\mathcal H}(x_j) 
= \delta_{ij} \mathbf{1} + \mathbf{0}.
\end{equation}
This means that 
\begin{equation}
\widehat \varphi_{k,i}(x_j)=\delta_{ij}\,, \mbox{ for all } x_j\in \mathcal I_k \cup \mathcal M_k \label{delta_property3}.
\end{equation}
Since the $ \widehat \varphi_{k,i}$ have the delta property, 
the set $\{ \widehat \varphi_{k,i}$ : $x_i \in \mathcal I_k \cup \mathcal M_k\}$, is linearly independent.
\end{proof}

\begin{definition}\label{def2}
Let $\widehat{N}_k$ be the total number of interior and master nodes. 
We define the set $\widehat{V}_k$ initially introduced in \eqref{fem_spaces} to be the set spanned by 
$\{ \widehat{\varphi}_{k,i}\}_{i=1}^{\widehat{N}_k}$, 
namely $\widehat{V}_k \equiv span(\{ \widehat{\varphi}_{k,i}\}_{i=1}^{\widehat{N}_k})$.
Because of Lemma \ref{hatphi_li}, it follows that $\{ \widehat{\varphi}_{k,i}\}_{i=1}^{\widehat{N}_k}$ 
also forms a basis for $\widehat{V}_k$, where $\widehat{V}_k$ is a vector space
with the standard addition and scalar multiplication for real valued functions.
\end{definition}

\begin{remark}\label{weird}
To the unique representation of an element $\widehat v \in \widehat{V}_k$
$$ \widehat v = \sum_{i: x_i \in \mathcal I_k \cup \mathcal M_k} \widehat{ \mathrm{v}}_i \widehat{\varphi}_{k,i}, $$
we artificially add the hanging node functions, which are identically zero, i.e.
$$ \widehat v = \sum_{i: x_i \in \mathcal I_k \cup \mathcal M_k} \widehat{ \mathrm{v}}_i \widehat{\varphi}_{k,i}
+ \sum_{i: x_i \in \mathcal H_k} \widehat{ \mathrm{v}}_i \widehat{\varphi}_{k,i}.
$$
The new representation of $\widehat v$ is not unique anymore, since the coefficients ${\mathrm{\widehat v}}_i$ 
associated with the zero hanging node functions are arbitrary. This choice may seem odd,
but we again emphasize that it is motivated by the fact that in the numerical implementation of the algorithm we want
to preserve the same dimensions between the arrays and matrices associated with the spaces $V_k$ and $\widehat V_k$.
\end{remark}

\begin{proposition}
The space $\widehat{V}_k$ is a subspace of $V_k$.
\end{proposition}
\begin{proof}
The proof is immediate, since each function $\widehat \varphi_{k,i}$ is constructed as a linear combination of 
functions $\varphi_{k,i}$, which belong to a basis for $V_k$. 
\end{proof}

\begin{proposition}\label{magheggio}
Let $a\in V_k$ and $\widehat b\in \widehat V_k$. Let 
$ \mathbf{a}=[\mathbf{a}_{\mathcal I}, \mathbf{a}_{\mathcal M},\mathbf{a}_{\mathcal H} ]^{\mathsf{T}} $
and $ \widehat{\mathbf{b}}=[\widehat{\mathbf{b}}_{\mathcal I}, \widehat{\mathbf{b}}_{\mathcal M},\widehat{\mathbf{b}}_{\mathcal H} ]^{\mathsf{T}} $
be the coefficient representation vectors of $a$ and $\widehat b$, i.e.
$${a} = {{\bm \varphi}^{\mathsf{T}}_k}{\mathbf{a}}, \quad\mbox{ and } \quad
\widehat{b} = {\widehat {\bm \varphi}_k}^{\mathsf{T}} \widehat {\mathbf{b}},$$
Then $a=\widehat b$ iff 
$\qquad
\mathbf{a}_{\mathcal I}=\widehat{\mathbf{b}}_{\mathcal I},\quad 
  \mathbf{a}_{\mathcal M}=\widehat{\mathbf{b}}_{\mathcal M}, \quad\mbox{and}\quad
  \mathbf{a}_{\mathcal H}=\Phi_k^{\mathsf{T}} {\mathbf{a}}_{\mathcal M}.
 $
\end{proposition}
\begin{proof}
The proof of the proposition follows from Eq. \eqref{matrix_vector}. 
\end{proof}
\begin{remark} \label{same_representation}
 Note that the above proposition implies that the equality of $a\in V_k$ and $\widehat b\in \widehat V_k$ is independent of the 
 value of $\widehat{\mathbf{b}}_{\mathcal H}$.
 Moreover if  $a\in V_k$ is such that $\mathbf{a} = [\mathbf{a}_{\mathcal I}, \mathbf{a}_{\mathcal M},\Phi_k ^T \mathbf{a}_{\mathcal M} ]^{\mathsf{T}}$,
 then $a$ is also in $\widehat V_k$ and we can choose the same vector $\mathbf{a}$ as representation of $a$ in ${\widehat V}_k$.
  
\end{remark}

\subsection{The inter-space operators}
\begin{definition}
 The prolongation operator $\widehat {\mathcal P}_k : \widehat V_k  \rightarrow V_k$ 
is the natural injection and its action on $\widehat v$ is given by 
$$\widehat {\mathcal P}_k \widehat v = {\bm \varphi}^{\mathsf{T}}_k \left(\widehat P_k \widehat {\mathbf{v}}\right),$$
\new{where $\widehat P_k$ is the matrix representation of $\widehat {{\mathcal P}}_k$.}
\end{definition}
\begin{proposition}
Let $\widehat v$ be given as a linear combination of the basis of $\widehat V_k$ as in Remark \ref{weird}. 
In vector notation $$\widehat v = {\widehat {\bm \varphi}_k}^{\mathsf{T}} \widehat { \mathbf{v} },$$
for some coefficient vector $\widehat {\mathbf{v}} = \left[\widehat{ \mathrm{v}}_1, \widehat {\mathrm{v}}_2\dots,\widehat {\mathrm{v}}_{N_k}\right]^{\mathsf{T}} \in \mathbb{R}^{N_k}$. 
Then, for the matrix representation of $\widehat {\mathcal P}_k$ we have 
$\widehat P_k = \widehat R_k^\mathsf{T}, $ where
$\widehat R_k$ is the matrix from Eq. \eqref{matrix_vector}.
\end{proposition}
\begin{proof}
Since $\widehat v \in \widehat V_k \subseteq V_k$,
its prolongation into $V_k$ is the natural injection, namely
$${\bm \varphi}^{\mathsf{T}}_k \left(\widehat P_k \widehat {\mathbf{v}}\right)=
\widehat{\mathcal{P}}_k \widehat v 
= \widehat v 
= {\widehat {\bm \varphi}_k}^{\mathsf{T}} \widehat {\mathbf{v}}
= \left({\widehat R_k {\bm \varphi}_k}\right)^{\mathsf{T}} \widehat {\mathbf{v}}  = 
 {\bm \varphi}^{\mathsf{T}}_k \left(\widehat R^{\mathsf{T}}_k \widehat {\mathbf{v}}\right), 
 \mbox{ for all } \widehat v \in \widehat V_k.$$
Thus, the matrix representation of $\widehat{\mathcal{P}}_k$ is given by $\widehat P_k = \widehat R^\mathsf{T}_k$.
\end{proof}

\begin{definition}
Let  $\langle \cdot , \cdot \rangle$ denote the $L^2(\Omega)$ inner product.
The restriction operator $\widehat{\mathcal R}_k : V_k \rightarrow \widehat V_k$ is defined as the adjoint
of $\widehat{\mathcal P}_k$ with respect to the $L^2(\Omega)$ inner product. In other words
$$\langle \widehat {\mathcal R}_k v,\widehat u\rangle = \langle v, \widehat {\mathcal P}_k \widehat u\rangle, \mbox{ for all } v\in V_k,\; \widehat u \in \widehat V_k.$$
Clearly, its matrix representation is the matrix $\widehat{R}_k$ from Eq. \eqref{matrix_vector}.
\end{definition}

Note that if we choose $\widehat u = \widehat {\varphi}_{k,i}$ for some $i$, then 
$$
\widehat {\mathcal P}_k \widehat \varphi_{k,i} 
= \widehat \varphi_{k,i}  
= \widehat R_{k,i}\, {\bm \varphi}_k,$$
where $\widehat R_{k,i}$ corresponds to the $i$-th row of the matrix $\widehat R_k$ defined in \eqref{matrix_vector}.
Consequently, for any $v$ in $V_k$
\begin{equation}  \label{restrictioni}
\langle \widehat {\mathcal R}_k v,\widehat \varphi_{k,i}\rangle 
= \langle v, \widehat {\mathcal P}_k \widehat \varphi_{k,i}\rangle 
=  \langle v, \widehat R_{k,i} \bm \varphi_{k}\rangle = \sum_{j=1}^{N_k} \widehat R _{k,ij} \langle v,\varphi_{k,j}\rangle.
\end{equation}
Let $g \in H^{-1}(\Omega)$ be given and define in an entry-wise fashion the vectors $\widehat {\mathbf{f}}$ and $\mathbf{f}$ 
$$ \widehat {\mathbf{f}}_i(g) = \langle g ,\widehat \varphi_{k,i}\rangle, \mbox{ and } 
   \mathbf{f}_i(g) = \langle g, \varphi_{k,i}\rangle, $$ for all $i=1,\cdots, N_k$. Then Eq.~\eqref{restrictioni}
   can be rewritten as
   $$ \widehat {\mathbf{f}}_i(\widehat {\mathcal R}_k v) = \sum_{j=1}^{N_k} \widehat R _{k,ij} \mathbf{f}_j(v) = \widehat R _{k,i} \mathbf{f}(v),$$ for all $i=1,\cdots, N_k$.
Then in matrix-vector notation, for any $v$ in $V_k$
$$\widehat {\mathbf{f}}(\widehat {\mathcal R}_k v) = \widehat R_{k} \mathbf{f}(v).$$

\begin{lemma}\label{matrixAhat}
For any bilinear form $a(\cdot,\cdot)$, define the matrices $A_k$ and $\widehat{A}_k$ by
\begin{align*}
& A_{k,{ij}} = a(\varphi_{k,i},\varphi_{k,j}), \mbox{ for all } i,\,j=1,\cdots,N_k,\\ 
& \widehat A_{k,{ij}} =  a(\widehat\varphi_{k,i},\widehat\varphi_{k,j}), \mbox{ for all } i,\,j=1,\cdots, N_k.
\end{align*}
Then  
$ \widehat A_k  = \widehat R_k A_k \widehat R_k^{\mathsf{T}}= \widehat R_k A_k \widehat P_k.$
Namely, with a slight abuse of terminology, $\widehat A_k$ is the restriction of $A_k$ to the space $\widehat V_k$.
\end{lemma}
\begin{proof}
For all $i,j$ 
\begin{align*}
\widehat A_{k,{ij}}& =  a(\widehat\varphi_{k,i},\widehat\varphi_{k,j}) = a(\widehat R_{k,i} \bm \varphi_{k},\widehat R_{k,j} \bm \varphi_{k}) 
= \sum_{l=1}^{N_k} \sum_{m=1}^{N_k} \widehat R_{k,il}  a( \varphi_{k,l}, \varphi_{k,m}) \widehat R_{k,jm} \\[-2ex]
&= \sum_{l=1}^{N_k} \sum_{m=1}^{N_k} \widehat R_{k,il} A_{k,lm} \widehat R_{k,jm} =
 \sum_{l=1}^{N_k} \sum_{m=1}^{N_k} \widehat R_{k,il} A_{k,lm} \widehat R^{\mathsf{T}}_{k,mj}=
\widehat R_{k,i} A_k \widehat R_{k,j}^\mathsf{T}, 
\end{align*}
thus the result follows.
\end{proof}

The rows and the columns corresponding to the hanging nodes are not removed but zeroed. 
In order to get a nonsingular matrix, the diagonal entries of $\widehat A_k$ corresponding to the hanging nodes are set equal to one.
%

Note that from Definition \ref{alg1} and \ref{alg2} it follows that 
$$V_1 \subseteq V_2 \subseteq \ldots \subseteq V_k \qquad \mbox{and} \qquad \widehat{V}_1 \subseteq \widehat{V}_2 \subseteq \ldots \subseteq\widehat{V}_k.$$
For this reason, prolongation operators can be defined.
\new{
\begin{definition}
The prolongation operator $\mathcal {Q}_k: V_{k-1} \rightarrow V_{k}$ is the natural injection between $V_{k-1}$ and $V_{k}$, while
the prolongation operator $ \widehat{\mathcal{Q}}_k: \widehat{V}_{k-1} \rightarrow \widehat{V}_{k}$ is the natural injection 
between $\widehat V_{k-1}$ and $\widehat V_{k}$. The respective matrix representations are denoted by
$Q_k$ and $\widehat{Q}_k$.
\end{definition}
Since $\widehat{V}_{k-1} \subseteq {V}_{k-1} \subseteq V_k $, it immediately follows that 
the prolongation operator $\mathcal {Q}_k\mathcal {\widehat P}_{k-1}:\widehat{V}_{k-1} \rightarrow V_{k}$ 
with matrix representation given by ${Q}_k {\widehat P}_{k-1}$ is the natural injection between 
$\widehat{V}_{k-1}$ and ${V}_k$.
}
\begin{lemma} \label{lemmaProj}
The natural injection from $\widehat{V}_{k-1}$ to $V_{k}$ can be obtained 
as a composition of the two prolongations 
$\mathcal {\widehat P}_{k-1}: \widehat{V}_{k-1} \rightarrow V_{k-1}$ and 
$\mathcal {Q}_k: {V}_{k-1} \rightarrow V_k$
$$ a = \mathcal {Q}_k(\mathcal {\widehat P}_{k-1}(\widehat {v}_{k-1})) = \widehat {v}_{k-1}, \mbox{ for all } \widehat {v}_{k-1} \in \widehat {V}_{k-1},$$
or in vector-matrix notation
$ \mathbf{a} =  {Q}_k {\widehat P}_{k-1}\widehat {\mathbf v}_{k-1},$
\new{where $\mathbf{a}$ is the vector representation of $a$ in $V_k$.}
\end{lemma}


\begin{proposition} 
The matrix \,${Q}_k {\widehat P}_{k-1}$ can be also used as matrix representation 
of the prolongation $\mathcal {\widehat Q}_k : \widehat V_{k-1} \rightarrow \widehat V_k$.
\end{proposition}
\begin{proof}
Let $\widehat {\mathbf v}_{k-1}$  be the vector representation of any $\widehat {v}_{k-1} \in \widehat {V}_{k-1}$. 
We want to show that if $\widehat{\mathbf{b}} \equiv {Q}_k {\widehat P}_{k-1} \widehat {\mathbf v}_{k-1}$, then the following equality holds,
$$ \widehat{b} \equiv {\widehat {\bm \varphi}_k}^T \widehat{\mathbf{b}} = \widehat {v}_{k-1}. $$
From Lemma \ref{lemmaProj}, the vector $\mathbf{a} = {Q}_k {\widehat P}_{k-1}\widehat {\mathbf v}_{k-1} $
is the vector representation of $a\in V_k$, with $a = \widehat {v}_{k-1}$.
Since $a = \widehat v_{k-1} \in \widehat{V}_{k-1} \subseteq \widehat{V}_{k} \subseteq V_k $, 
$a$ is also an element of $\widehat{V}_{k}$.
According to Remark \ref{same_representation} we can 
choose $\mathbf{a}$ as vector representation of $a$ in $\widehat V_k$.

Then, we have two elements $\widehat b$ and $a$ in $\widehat V_k$ with the same vector representation, 
$\widehat{\mathbf{ b}} = \mathbf{a}$, thus they are equal. This means $\widehat b = a = \widehat v_{k-1}$.
\end{proof}

Similarly, as we did before for the nested spaces $V_k$ and $\widehat V_k$, we can show that the matrix representation
of the restriction operator between $\widehat V_k$ and $\widehat V_{k-1}$ is $\widehat Q_k^{\mathsf{T}}$,
the transpose of the matrix representation of the prolongation operator.

\begin{lemma}
For any bilinear form $a(\cdot,\cdot)$, let $\widehat{A}_k$ be the matrix defined in Lemma \ref{matrixAhat} and let $\widehat{A}_{k-1}$ be
$$\widehat A_{{k-1},{ij}} =  a(\widehat\varphi_{{k-1},i},\widehat\varphi_{{k-1},j}), \mbox{ for all } i,\,j=1,\cdots, N_{k-1}.$$
Then,  
$ \widehat A_{k-1} = \widehat Q_k^{\mathsf{T}} \widehat A_k \widehat Q_k,$
namely, with a slight abuse of terminology, $\widehat A_{k-1}$ is the restriction of $\widehat A_{k}$ to the space $\widehat V_{k-1}$.

\end{lemma}
\begin{proof}
 The proof is similar to the one in Lemma \ref{matrixAhat}.
\end{proof}

Once again, the diagonal entries of $\widehat A_{k-1}$ corresponding to the hanging nodes are set equal to one.

Schematics of all the inter-space operators defined above, and how to use them, are given in Figure \ref{mg1}. 
\new{All the prolongation and restriction operators are global and are implemented in the FEMuS library \cite{femus-web-page} using PETSc sparse parallel matrices \cite{petsc-web-page}. The restriction operator $\widehat{R}_k$ from $V_k$ to $\widehat{V}_k$ is pre-computed when the irregular mesh is generated since it depends only on the mesh and the finite element family. Similarly, the prolongation operator $Q_k$ from $V_{k-1}$ to $V_k$ is pre-computed once the meshes at the two consecutive refinement levels are available, since it also depends only on the mesh and the finite element family.
The prolongation operator $\widehat{Q}_k$ is given by the product between the two matrices and is pre-computed as well.}
\begin{figure}[!t]
\centering
\includegraphics[scale=0.45]{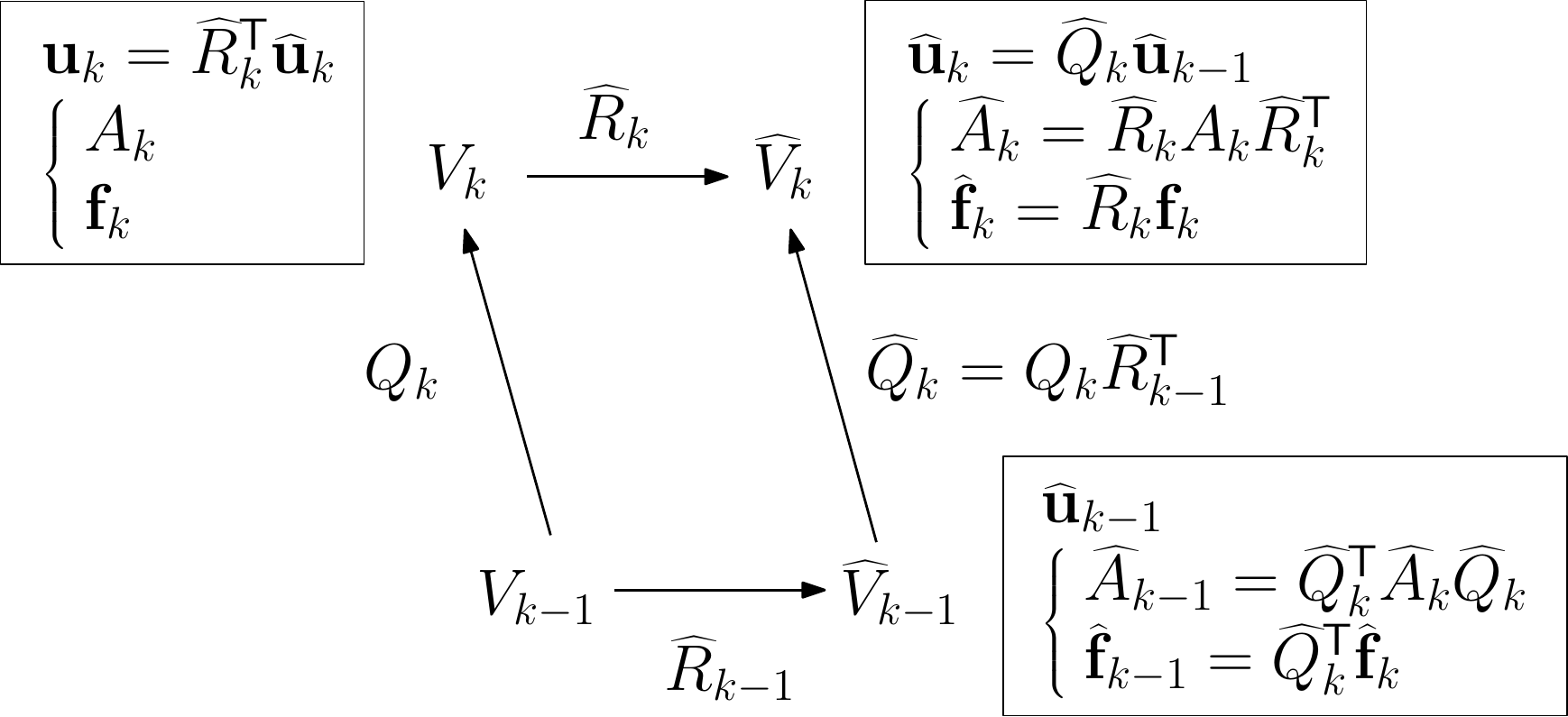}
\caption{Inter-space operators }\label{mg1}
\end{figure}

\subsection{The numerical algorithm} \label{num_alg}
Once the inter-space operators are available in terms of matrices, the numerical solution of the problem
$ a(u,v) = \langle f,v\rangle,$
is sought in the continuous space $\widehat V_J$. Namely, we seek the solution of the discretized problem
$ \widehat A_J \widehat{\mathbf{u}}_J = \widehat {\mathbf f}_J.$
We do not work directly in the continuous space $\widehat V_J$.
We rather work in the discontinuous space $V_J$ 
and use the restriction operator $\widehat R_J$ to move objects between spaces. 
Since an iterative solver is used the above system is rewritten in its residual form. 

For $i=0$, let $\widehat{\mathbf{u}}_J^0$ be the initial guess, then
\begin{enumerate}[leftmargin= * ]
\item assemble the matrix $A_J$ in $V_J$ using the test functions $\bm \varphi_{J}$;
\item restrict $A_J$ in the $\widehat V_J$ space
$$\widehat A_J = \widehat R_J A_J \widehat R_J^{\mathsf{T}};$$
\item prolong the current solution in the $V_J$ space
$$\mathbf{u}_J^i = \widehat R_J^{\mathsf{T}}\widehat{\mathbf{u}}_J^i;$$
\item assemble the residual in $V_J$
$$\mathbf r_J^i = \mathbf f_J - A_J \mathbf{u}_J^i;$$
\item restrict the residual to $\widehat V_J$
$$\widehat{\mathbf r}_J^i = \widehat{R}_J \mathbf r_J^i; $$
\item Let $\widehat{D}_J$ be an ``\textit{easy to invert}'' approximation of the matrix $\widehat A_J$, then 
solve $$ \widehat{D}_J \widehat{\mathbf{w}}_J^{i} = \widehat{\mathbf{r}}_J^{i} \mbox{, and set }
\widehat{\mathbf{u}}_J^{i+1} = \widehat{\mathbf{u}}_J^{i} + \widehat{\mathbf{w}}_J^{i}; $$
\item if $\|\widehat {\mathbf{w}}_J^{i}\|\le \varepsilon$ exit, otherwise 
set $i=i+1$ and go back to step 3.
\end{enumerate}
A few remarks on the above algorithm:

We construct the matrix $A_J$ and the residual $\mathbf r_J$ in $V_J$, since the assembling is done element-wise following
the standard finite element approach for unstructured grids. Moreover, let $x_i$, $x_j$ in $X_J^l$ be two nodal
points belonging to a generic element $E$ of the triangulation $\mathcal T_J^l$. The two pairs of test functions 
$\varphi_{J,i},\varphi_{J,j}$ and $\varphi^l_{J,i},\varphi^l_{J,j}$ 
differ at most on the boundary of the element $E$.  Then on $E$
\begin{align*}
a(\varphi_{J,i},\varphi_{J,j})\big|_{E} = a(\varphi^l_{J,i},\varphi^l_{J,j})\big|_{E},&&\mbox{and }&&
 \langle f, \varphi_{J,i} \rangle\big|_{E} = \langle f, \varphi^l_{J,i} \rangle\big|_{E},
\end{align*}
for all $f\in H^{-1}(\Omega)$. Thus to construct $A_J$ and $\mathbf r_J$ element-wise we use the standard 
test functions $\varphi^l_{J,i}$, which is a very convenient approach since no new bases have to be really constructed.

The operator $\widehat D_J$ is abstract, and its inversion can be interpreted as the action of a linear algebra solver 
which in turn can also be preconditioned.
In our applications, we used geometric multigrid (MG) either as a solver (in Section \ref{eigAnalysis}) 
or as a preconditioner for the conjugate gradient (CG) method and the generalized minimal residual (GMRES) (in Sections \ref{numex1} and \ref{numex2}).
For this purpose we used the C++ library PETSc \cite{petsc-web-page} as a black box, 
providing  as inputs the matrices $\widehat A_k$ for $k=J,\dots,0$, the matrices $\widehat Q_k$ for $k=J,\dots,1$ 
and the residual vector $\widehat{\mathbf r}^i_J$,
and letting PETSc compute $\widehat{\mathbf w}_J^i$. This algebraic approach is convenient and elegant since all the burden 
of dealing with the hanging nodes at different level triangulations $\mathcal T_k$ has been incorporated in 
the prolongation operator $\widehat Q_k$
and no ad-hoc multigrid solver or preconditioner needs to be implemented.

The above algorithm can be easily  extended to nonlinear problems. 
In this case the residual $\mathbf r_J^i$ at step 4 is a more complex function of $\mathbf u_J^i$ and $\widehat D_J$ is 
an ``\textit{easy to invert}'' approximation of the tangent matrix  
$$\widehat J_J(\widehat {\mathbf u}_J^i) = 
\widehat R_J \left( -\frac{\partial  {\mathbf r}_J^i}{\partial {\mathbf u}_J^i}\right) \widehat R_J^{\mathsf{T}}.$$
\new{This corresponds to a Newton-Raphson iterative scheme \cite{ke2017block}}. Such a scheme is used in the last example 
where a Navier-Stokes flow is considered.

\section{Eigenvalue analysis of the multigrid method}\label{eigAnalysis}
Historically, the theory of multigrid methods has been developed for linear equations of the form
$A u = f,$
where the operator $A$ is symmetric and positive definite (SPD)
\cite{hackbusch2013multi, bramble1993multigrid, yserentant1993old, bramble1987new, xu1992iterative, bramble1991convergence}.
Other types of boundary value problems have also been addressed, where $A$ is nonsymmetric or indefinite
\cite{olshanskii2004convergence, reusken2002convergence, wu2006analysis, bramble1988analysis}.

In \cite{bramble1991convergence}, the authors obtained convergence estimates for a multigrid algorithm without making 
regularity assumptions on the solution. 
This was a breakthrough in the theoretical analysis of multigrid, since previous convergence estimates used to rely on 
both a smoothing property, and an approximation property,
where the latter is typically obtained assuming a certain degree of regularity.
Moreover, in the same paper, it was shown that convergence of the multigrid algorithm could be obtained also when smoothing
is performed only on a subspace of the multigrid space at every level.
This allows applying the theory proposed in \cite{bramble1991convergence} to local midpoint refinement applications,
where the smoothing is performed on subspaces that do not have degrees of freedom associated with the interface nodes.
This includes both hanging and master nodes.
A weaker convergence estimate with respect to the classical multigrid convergence \cite{brenner2007mathematical} was obtained, in the sense that the error bound $\delta_J$ for the multigrid error $E_J$, depended on the number
of multigrid levels $J$, 
\begin{align*}
& a(E_Jv,v) \le \delta_J a(v,v) , \mbox{ with } \delta_J = 1-(C\,J)^{-1},
\end{align*}
where is $a(\cdot,\cdot)$ is the energy norm and $C$ is a constant independent of $J$.
It is clear when looking at $\delta_J$ that, from a theoretical point of view, increasing the pre and/or post-smoothing iterations at each level does not guarantee an improvement of the convergence bound.
From a practical point of view as well, since smoothing is done only on subspaces of the actual multigrid spaces,
an increase in the number of smoothing steps can only improve the convergence bound up to some saturation value.
In this work, using the continuous nodal basis functions in Definition \ref{alg2}, we are allowed to perform
the smoothing procedure on the entire multigrid space. This results in an improved convergence when the number of smoothing
steps is increased, as it will be shown from the numerical results.
For the rest of the paper, we refer to the local smoothing approach outlined in \cite{bramble1991convergence} as BPWX algorithm after the authors: Bramble, Pasciak, Wang and Xu.

\begin{figure}[!t] 
\begin{minipage}{1.0\textwidth}
\begin{center}
\begin{minipage}{0.32\textwidth}
\begin{center}
\includegraphics[width=.8\linewidth]{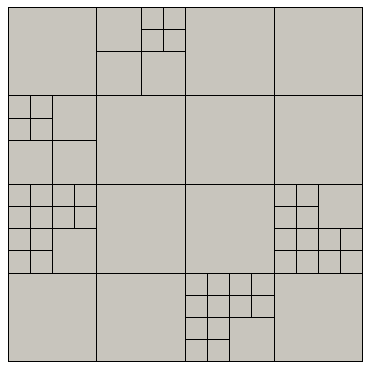}\\
(a) level 3 
\end{center}
\end{minipage}
\begin{minipage}{0.32\textwidth}
\begin{center}
\includegraphics[width=.8\linewidth]{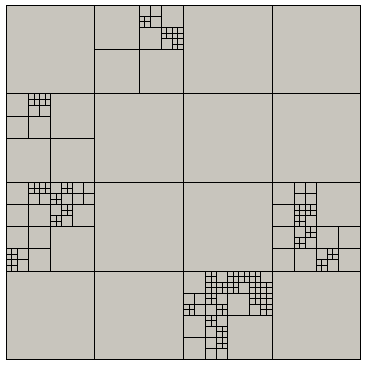} \\
(b) level 5
\end{center}
\end{minipage}
\begin{minipage}{0.32\textwidth}
\begin{center}
\includegraphics[width=.8\linewidth]{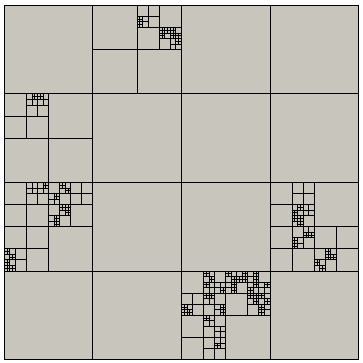}\\
(c) level 6 
\end{center}
\end{minipage}
\end{center}
\end{minipage}
\caption{Refinement of the quadrilateral element mesh: 50\% of the elements from the previous level are randomly selected for refinement.} 
\label{fig:quad_refinement}
\end{figure}

\begin{figure}[!b] 
\begin{minipage}{1.0\textwidth}
\begin{center}
\begin{minipage}{0.32\textwidth}
\begin{center}
\includegraphics[width=.8\linewidth]{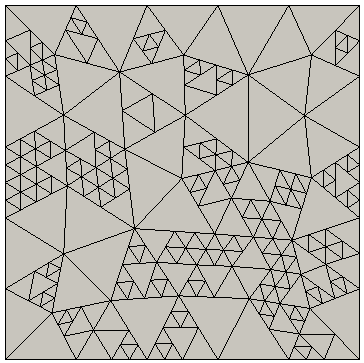}\\
(a) level 2 
\end{center}
\end{minipage}
\begin{minipage}{0.32\textwidth}
\begin{center}
\includegraphics[width=.8\linewidth]{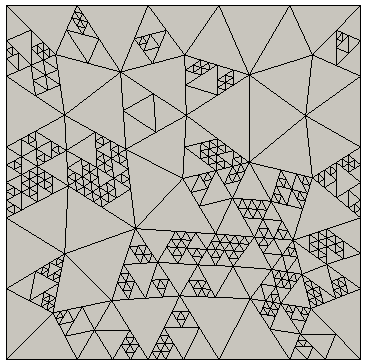} \\
(b) level 3 
\end{center}
\end{minipage}
\begin{minipage}{0.32\textwidth}
\begin{center}
\includegraphics[width=.8\linewidth]{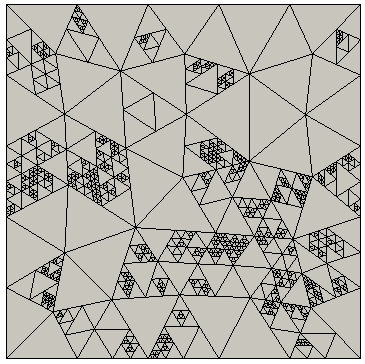}\\
(c) level 4
\end{center}
\end{minipage}
\end{center}
\end{minipage}
\caption{Refinement of the triangular element mesh: 50\% of the elements from the previous level are randomly selected for refinement.} 
\label{fig:tri_refinement}
\end{figure}

In this section, we compare the performance of our algorithm with  the BPWX, and show that
since our multigrid performs a global smoothing on all nodes at a given multigrid level,
the drawbacks of the BPWX method are eliminated.

We do so by evaluating the spectral radius of $E_J = I-\widehat D_J^{-1} \widehat A_J$, where the matrix $\widehat D_J^{-1}$ 
corresponds to one V-cycle of either the BPWX multigrid or the one proposed in this work.
For this purpose, we use the method developed in \cite{oosterlee1998evaluation}. 

\begin {table}[!t]
\caption{
\new{Spectral radius considering the meshes depicted in Figure \ref{fig:quad_refinement} and Figure\ref{fig:tri_refinement}}.
One pre and post-smoothing iteration is considered.}
\begin{center}
	\begin{tabular}{|c|c|c|c|c|} \hline	
	Methods & \multicolumn{2}{|c|}{Proposed Method} 
	        & \multicolumn{2}{|c|}{BPWX Method} \\ \hline 
	     ~L~ & ~Quad.~ & ~Tri.~ & ~Quad.~ & ~Tri.~ \\ \hline   
	     (a) & 1.75e-2 & 5.35e-2 & 1.37e-1 & 3.03e-1 \\ \hline 
	     (b) & 2.25e-2 & 5.79e-2 & 1.97e-1 & 3.04e-1 \\ \hline 
	     (c) & 1.74e-2 & 7.87e-2 & 2.04e-1 & 3.04e-1 \\ \hline 
	\end{tabular}
\end{center}
\label{Spectral_levels}
\end{table}

Consider one iteration for the algorithm described in Section \ref{num_alg}.
The solution at the first iteration step is given by 
\[\widehat{\vec{u}}_J^{1} = \widehat{\vec{u}}_J^{0} + \widehat{\vec{w}}_J^1 
= (I-\widehat D_J^{-1}\widehat A_J)\,\widehat{\vec{u}}_J^{0} +  \widehat D_J^{-1} \widehat{\vec{f}}_J. \]
Let $ \widehat{\vec{u}}_J^{0} $ be zero, impose zero Dirichlet boundary conditions for $\widehat{\vec{u}}_J^{1}$, and set 
$ \widehat{\vec{f}}_J = \vec{e}_{l}$, with entries one at the $l$-th row and zero elsewhere. Then, the system becomes 
\[\widehat{\vec{u}}_{J,l}^{1} = \widehat D_J^{-1} \vec{e}_{l},~ \mbox{ for } l=1,\dots, N_J.\]
The inverse of the multigrid matrix is obtained using the solution $\widehat{\vec{u}}_{J,l}^{1}$ as $l$-th column of $\widehat D_J^{-1}$.
The spectrum of $\widehat D_J^{-1} \widehat A_J$ follows.

We remark that in the following examples only one iteration of the multigrid V-cycle is carried out.
At each level, we employ a Richardson smoother with damping factor 0.6, preconditioned with ILU.
Tests are made considering the bilinear operator $a(u,v) = \langle \nabla u , \nabla v \rangle$.
We use biquadratic and triquadratic Lagrangian test functions in 2D and 3D, respectively,
and study the performances of the two methods
with respect to different element types and 
different numbers of pre and post-smoothing iterations.

Convergence of the multigrid solver is obtained only if the spectral radius of $I-\widehat D_J^{-1} \widehat A_J$ is less than one. Moreover, the smaller such a quantity is, the better $\widehat D_J^{-1}$ approximates the inverse of $A_J$ and the faster the solver converges to the solution.

\begin {table}[!b]
\caption{Spectral radius for several numbers of pre and post-smoothing iterations. 
The chosen refined configuration corresponds
to case (c) for both mesh types in Figure \ref{fig:quad_refinement} and Figure\ref{fig:tri_refinement}.}
\begin{center}
	\begin{tabular}{|c|c|c|c|c|} \hline	
	Methods & \multicolumn{2}{|c|}{Proposed Method} 
	        & \multicolumn{2}{|c|}{BPWX Method} \\ \hline 
	     ~Smoothing~ & ~Quad.~ & ~Tri.~ & ~Quad.~ & ~Tri.~ \\ \hline   
	     1 & 1.74e-2 & 7.87e-2 & 2.04e-1 & 3.04e-1 \\ \hline 
	     2 & 5.80e-2 & 2.10e-2 & 2.02e-1 & 2.62e-1 \\ \hline 
	     4 & 2.10e-2 & 3.90e-3 & 2.02e-1 & 2.50e-1 \\ \hline 
	     8 & 9.16e-4 & 1.00e-3 & 2.02e-1 & 2.47e-1 \\ \hline 
	\end{tabular}
\end{center}
\label{Spectral_smoothing}
\end{table}

We first consider two types of elements for the two-dimensional unit square geometry: quadrilaterals and triangles. 
The criterion for the local selective refinement is that 50\% of the elements from the previous level are randomly selected for refinement.
For each element type we study the three irregular triangulations depicted in Figures \ref{fig:quad_refinement}    
and \ref{fig:tri_refinement}. 
In the quadrilateral case, the coarse mesh consists of a $2\times2$ regular triangulation, that is uniformly refined once and then 
selectively refined $2$, $4$ and $5$ times, for a total of $323$, $1201$ and $2265$ degrees of freedom, 
respectively. See cases (a), (b) and (c) in Figure \ref{fig:quad_refinement}.    
In the triangular case, the coarse mesh has $62$ elements 
that are selectively refined $2$, $3$ and $4$ times, for a total of $1145$, $2349$ and $4603$ degrees of freedom, respectively.
See cases (a), (b) and (c) in Figure \ref{fig:tri_refinement}.     

Both our method and the BPWX method are tested. 
Table \ref{Spectral_levels} shows the spectral radius of $I-\widehat D_J^{-1} \widehat A_J$ for the six cases in Figures \ref{fig:quad_refinement}    
and \ref{fig:tri_refinement},  when only one pre and post-smoothing iterations are used. 
As the number of levels increases, the spectral radius has slight variations but does not change dramatically 
for both quadrilateral and triangular elements. 
More importantly, at each level, with the same element type, the spectral radius of the new method 
is 3 to 11 times smaller than the one in the BPWX method. 
Roughly speaking this indicates that our solver would converge 3 to 11 times faster than the BPWX. 

Next we investigate the effect of the number of pre and post-smoothing iterations on the spectral radius at level c, see Table \ref{Spectral_smoothing}. 
Here we consider 1, 2, 4 and 8 symmetric pre and post-smoothing iterations.
With the proposed method, under the same element type, the spectral radius reduces 3 to 5 times as we double the number of iterations,  
while in the BPWX method it quickly saturates. 
Thus we expect that, with the proposed method, the number of linear iterations required by the solver to converge would decrease significantly,
while they would stay the same for the BPWX method.

We conclude that for two-dimensional examples, 
our method always has an advantage over the BPWX in terms of convergence rate of the solver. 
Moreover, this advantage increases dramatically as the number of smoothing iterations increases.

\begin{figure}[!t] 
\begin{minipage}{1.0\textwidth}
\begin{center}
\begin{minipage}{0.49\textwidth}
\begin{center}
\includegraphics[width=.7\linewidth]{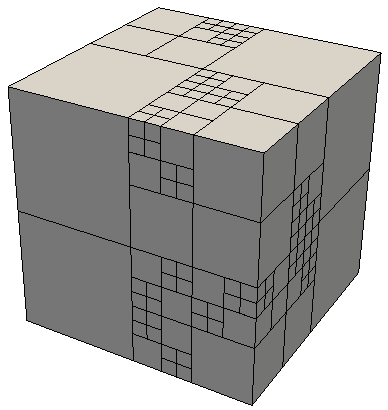}\\
hexahedral elements 
\end{center}
\end{minipage}
\begin{minipage}{0.49\textwidth}
\begin{center}
\includegraphics[width=.7\linewidth]{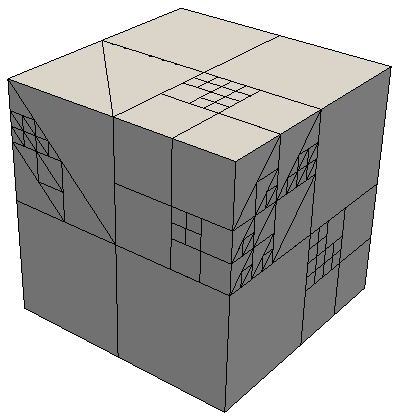}\\
mixed elements 
\end{center}
\end{minipage}
\end{center}
\end{minipage}
\caption{Refinement of different types of elements starting with one uniform level and refining up to three nonuniform levels: 
50\% of the elements from the previous level are randomly selected for refinement.} 
\label{fig:3D_refinement}
\end{figure}
\begin {table}[!b]
\caption{Spectral radius for different refined meshes.
One pre and one post-smoothing iteration is considered.}
\begin{center}
	\begin{tabular}{|c|c|c|c|c|c|c|c|c|} \hline	
	 & \multicolumn{4}{|c|}{Proposed Method} 
	        & \multicolumn{4}{|c|}{BPWX Method} \\ \hline 
	     L & Hex. & Wedge & Tet. & Mixed &  Hex. & Wedge & Tet. & Mixed \\ \hline   
	     1 & 7.50e-3 & 3.95e-2 & 2.31e-2 & 3.58e-2 & 9.34e-2 & 2.08e-1 & 6.96e-2 & 1.59e-1 \\ \hline
	     2 & 1.61e-2 & 5.71e-2 & 5.80e-2 & 1.01e-1 & 2.55e-1 & 2.07e-1 & 1.19e-1 & 2.29e-1 \\ \hline 
	     3 & 1.76e-2 & 6.25e-2 & 6.00e-2 & 1.12e-1 & 2.63e-1 & 2.98e-1 & 1.37e-1 & 3.17e-1 \\ \hline 
	\end{tabular}
\end{center}\label{Spectral_levels3d}
\end{table}

We then investigate a three-dimensional example in a unit box geometry 
considering triangulations made of four types of elements: hexahedra, wedges, tetrahedra and a combination of the three.
The criterion for the local selective refinement remains the same as the two-dimensional case. 
The coarse grid consists of $8$, $16$, $6$ and $20$ elements for 
the hexahedra, wedges, tetrahedra and the combination of the three, respectively.
Each coarse grid is selectively refined $1$, $2$ and $3$ times. 
The total number of degrees of freedom is $401$, $1389$ and $5087$ for the hexahedral case,
$506$, $1817$, $6538$ for the wedge-shaped case, $146$, $456$, $1684$ for the tetrahedral case, 
and $656$, $2306$, $8560$ for the mixed case.
Figure \ref{fig:3D_refinement} shows the three-dimensional irregular triangulations for the hexahedral and mixed cases at level $3$.

The results for the spectral radius with one pre and post-smoothing iteration are listed in 
Table \ref{Spectral_levels3d} for both the proposed method and the BPWX method.
Notice that the spectral radius under the same type of elements 
does not change significantly from level $2$ to $3$.
In some cases a larger variation occurs from level $1$ to $2$,   
that we attribute to the presence of relatively few degrees of freedom at level $1$. 
At each level, considering the same element type, the spectral radius of the new method 
is 2 to 15 times smaller than the one obtained with the BPWX method. 
Roughly speaking this indicates that our solver would converge 2 to 15 times faster than the BPWX one. 
In particular the hexahedral case performs the best, while the mixed is 2-4 times better.

Next we investigate the effect of the number of pre and post-smoothing iterations on the spectral radius at level $3$, see Table \ref{Spectral_smoothing3d}. Again we consider 1, 2, 4 and 8 symmetric pre and post-smoothing iterations.
Under the same element type in our method the spectral radius reduces 2 to 5 times doubling the number of iterations, 
while in the BPWX method it is invariant. 
Thus, we expect that the number of linear iterations required by solver to converge would decrease significantly with the proposed method, while it would stay the same for the BPWX method.
Once again the hexahedra perform the best, followed by the wedges, the tetrahedra and the mixed ones.
\begin {table}[!t]
\caption{Spectral radius for several numbers of pre and post-smoothing iterations. 
The chosen refined configuration corresponds to level 3 for all mesh types.}
\begin{center}
	\begin{tabular}{|c|c|c|c|c|c|c|c|c|} \hline	
	 & \multicolumn{4}{|c|}{Proposed Method} 
	        & \multicolumn{4}{|c|}{BPWX Method} \\ \hline 
	     S &Hex.& Wedge& Tet. & Mixed &  Hex. & Wedge & Tet. & Mixed \\ \hline   
	     1 & 1.76e-2 & 6.25e-2 & 6.00e-2 & 1.12e-1 & 2.63e-1 & 2.98e-1 & 1.37e-1 & 3.17e-1 \\ \hline
	     2 & 7.20e-3 & 2.23e-2 & 1.53e-2 & 3.45e-2 & 2.63e-1 & 2.65e-1 & 1.29e-1 & 3.16e-1 \\ \hline 
	     4 & 3.00e-3 & 1.05e-2 & 6.20e-3 & 2.08e-2 & 2.63e-1 & 2.61e-1 & 1.29e-1 & 3.16e-1 \\ \hline 
	     8 & 6.64e-4 & 3.40e-3 & 1.80e-3 & 1.06e-2 & 2.63e-1 & 2.61e-1 & 1.29e-1 & 3.16e-1 \\ \hline 
	\end{tabular}
\end{center}
\label{Spectral_smoothing3d}
\end{table}
%
%
%

\begin{figure}[!b]
\begin{center}
\begin{minipage}[t]{0.45\linewidth}
\centering
\includegraphics[width=.65\linewidth]{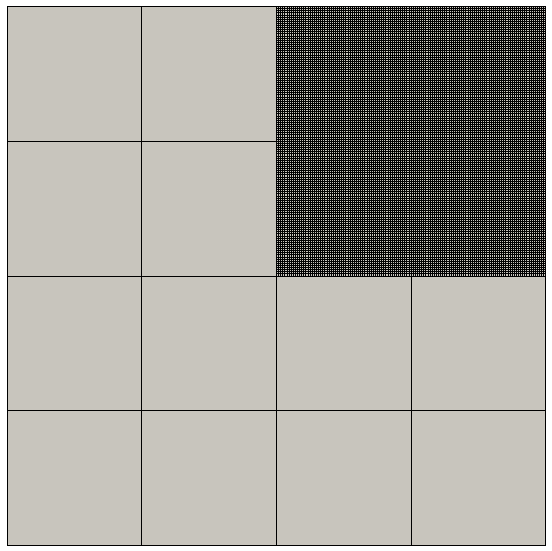}
\caption{Refinement of the first quadrant after 7 refinements.} 
\label{fig:quadrant_refinement}
\end{minipage} 
\hspace{0.125in}
\begin{minipage}[t]{0.45\linewidth}
\centering
\includegraphics[width=.65\linewidth]{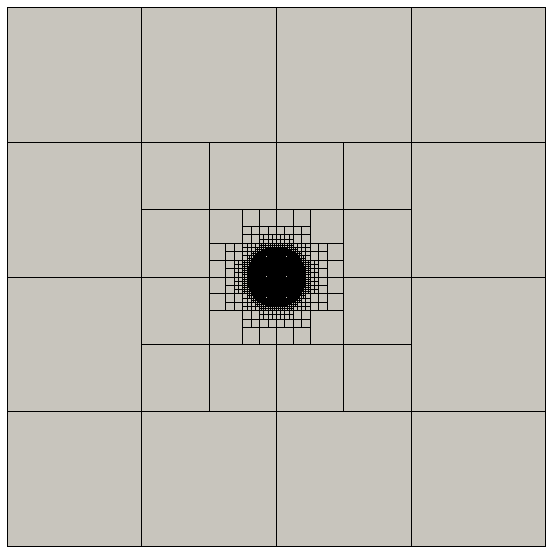}
\caption{Refinement of the circle after 8 refinements.} 
\label{fig:circle_refinement}
\end{minipage}
\end{center}
\end{figure}
\section{The Poisson problem}\label{numex1}
We continue our investigation by considering the weak formulation of the Poisson problem discussed in \cite{janssen2011adaptive}.
Namely find $u_h \in {V}_h \subset H_0^1(\Omega)$ such that
\begin{equation}
\label{Poisson_weak} 
(\nabla u_h, \nabla v_h) = (f,v_h), 
\mbox{ for all } v_h \in {V}_h,
\end{equation}
with $\Omega = [-1, 1]\times[-1,1]$, and $f \equiv 1$.
The space ${V}_h$ denotes an appropriate finite element space.
Note that the multigrid method for irregular triangulations proposed in \cite{janssen2011adaptive}
\new{only applies to $1$-irregular grids and has been implemented
only on quadrilateral and hexahedral element}. 
Our method has no such limitations.  

We start with a coarse grid consisting of $4$ elements, which are uniformly refined once.
Two selective refinement strategies are used:
refine only the elements in the first quadrant ($x \geq 0, y \geq 0 $),
and refine all the elements for which the centers are located within the circle of radius
$\frac{\pi}{4 k}$ ($k = 1,...,J-1$) centered at the origin. 
Figure \ref{fig:quadrant_refinement} shows the irregular triangulation of the first quadrant
for  $J=7$. Figure \ref{fig:circle_refinement} shows the irregular triangulation of the circle
for  $J=8$.

%
%
%
%
%
\begin {table}[!t]
\caption{Number of iterations and average logarithmic convergence rate for the Poisson problem.}
\begin{center}
	\begin{tabular}{|c|c|c|c|c|c|c|c|c|} \hline	
	\multirow{2}{*} {L}   & \multicolumn{2}{|c|}{Quadrant} & \multicolumn{2}{|c|}{Quadrant \cite{janssen2011adaptive}} 
			      & \multicolumn{2}{|c|}{Circle} & \multicolumn{2}{|c|}{Circle \cite{janssen2011adaptive}} \\ \cline {2-9} 
				& $n_{10}$ & $\bar{r}$ & $n_{10}$& $\bar{r}$ & $n_{10}$ & $\bar{r}$ & $n_{10}$ & $\bar{r}$ \\ \hline 
	     3 & 6 & 1.72 & 4 & 2.56 & 6 & 1.94 & 5 & 2.29 \\ \hline 
	     4 & 7 & 1.55 & 6 & 1.69 & 6 & 1.67 & 6 & 1.83 \\ \hline 
	     5 & 7 & 1.47 & 7 & 1.61 & 7 & 1.62 & 6 & 1.76 \\ \hline 
	     6 & 7 & 1.46 & 7 & 1.60 & 7 & 1.57 & 7 & 1.44 \\ \hline 
	     7 & 8 & 1.41 & 7 & 1.60 & 7 & 1.53 & 7 & 1.61 \\ \hline 
	     8 & 8 & 1.39 & 7 & 1.60 & 7 & 1.50 & 7 & 1.57 \\ \hline 
	     9 & 8 & 1.38 & 7 & 1.60 & 7 & 1.50 & 7 & 1.59 \\ \hline
	    10 & 8 & 1.37 & 7 & 1.60 & 7 & 1.52 & 7 & 1.59 \\ \hline
	\end{tabular}
\end{center}
\label{possion_interation}
\end{table}
\begin {table}[!b]
\caption{Comparison of different degrees of finite element families for the Poisson problem.}
\setlength\tabcolsep{5.5pt} 
\begin{center}
	\begin{tabular}{|c|c|c|c|c|c|c|c|c|c|c|c|c|} \hline	
	\multirow{3}{*} {L}   & \multicolumn{6}{|c|}{Quadrant} & \multicolumn{6}{|c|}{Circle} \\ \cline {2-13}
	& \multicolumn{2}{|c|} {Bilinear} & \multicolumn{2}{|c|} {Quadratic} & \multicolumn{2}{|c|} {Biquadratic} 
	& \multicolumn{2}{|c|} {Bilinear} & \multicolumn{2}{|c|} {Quadratic} & \multicolumn{2}{|c|} {Biquadratic} \\ \hline		      
	& $n_{10}$ & $\bar{r}$ & $n_{10}$& $\bar{r}$ & $n_{10}$ & $\bar{r}$ & $n_{10}$ & $\bar{r}$ & $n_{10}$ & $\bar{r}$ & $n_{10}$ & $\bar{r}$ \\ \hline 
	     3 & 6 & 1.72 & 7 & 1.52 & 7 & 1.57 & 6 & 1.94 & 7 & 1.56 & 7 & 1.61 \\ \hline 
	     4 & 7 & 1.55 & 7 & 1.51 & 7 & 1.46 & 6 & 1.67 & 7 & 1.51 & 7 & 1.51 \\ \hline 
	     5 & 7 & 1.47 & 7 & 1.48 & 8 & 1.41 & 7 & 1.62 & 7 & 1.50 & 7 & 1.48 \\ \hline 
	     6 & 7 & 1.46 & 7 & 1.46 & 8 & 1.39 & 7 & 1.57 & 7 & 1.50 & 7 & 1.49 \\ \hline 
	     7 & 8 & 1.41 & 7 & 1.44 & 8 & 1.39 & 7 & 1.53 & 7 & 1.52 & 7 & 1.56 \\ \hline 
	     8 & 8 & 1.39 & 7 & 1.43 & 8 & 1.39 & 7 & 1.50 & 7 & 1.55 & 7 & 1.57 \\ \hline 
	     9 & 8 & 1.38 & 8 & 1.44 & 8 & 1.38 & 7 & 1.50 & 7 & 1.57 & 7 & 1.57 \\ \hline
	    10 & 8 & 1.37 & 8 & 1.42 & 8 & 1.37 & 7 & 1.52 & 7 & 1.60 & 7 & 1.61 \\ \hline
	\end{tabular}
\end{center}
\label{possion_interation_degree}
\end{table}

Continuous piecewise-bilinear approximation is considered for both $u_h$ and $v_h$.
The linear system arising from the weak formulation in Eq.\eqref{Poisson_weak} 
defined on the above irregular triangulations is solved 
using a Conjugate Gradient (CG) solver, preconditioned with the proposed V-cycle Multigrid.
At each level, we employ a Richardson smoother with one iteration for pre and post-smoothing and damping factor 0.8.
The smoother is further preconditioned with ILU. 
In Table \ref{possion_interation}, we compare our results with the ones in \cite{janssen2011adaptive}.
The multigrid smoother in \cite{janssen2011adaptive} is a symmetric multiplicative Schwarz smoother \cite{xu1992iterative} and the smoothing procedure is performed only
locally.
We present the number of CG steps $n_{10}$ needed to reduce the norm of the residual $r$ by a factor of $10^{10}$, 
and the average logarithmic convergence rate according to Janssen \cite{janssen2011adaptive} and Varga \cite{varga2009matrix}:
$$ \overline{r} = \frac{1}{n} \log_{10} \frac{|{r_0}|}{|{r_n}|},  $$
where $|{r_n}|$ is the Euclidean norm of the residual vector $r_n$ at the $n-$th CG step.
We observe that as the local refinement level increases, the values of both $n_{10}$ and $\overline{r}$ saturate and 
are very close to the results from Janssen \cite{janssen2011adaptive}.
No significant degeneration is observed showing 
that our method is suitable for more general applications with respect to the $1$-irregular grid constraint.

Table \ref{possion_interation_degree} shows more results for various finite element families: bilinear, quadratic and biquadratic.
For both refinements of the  circle and  of the first quadrant, the number of iterations and average logarithmic convergence rate 
are almost independent of the refinement level and of the degree of the finite element family.

\begin{figure}[!b] 
\begin{minipage}{1.0\textwidth}
\begin{center}
\begin{minipage}{0.42\textwidth}
\begin{center}
\includegraphics[width=.75\linewidth]{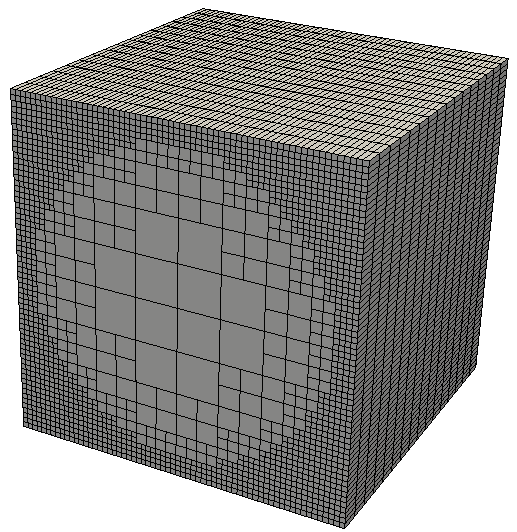}\\
(a) level-$4$ grid
\end{center}
\end{minipage}
\begin{minipage}{0.53\textwidth}
\begin{center}
\includegraphics[width=.75\linewidth]{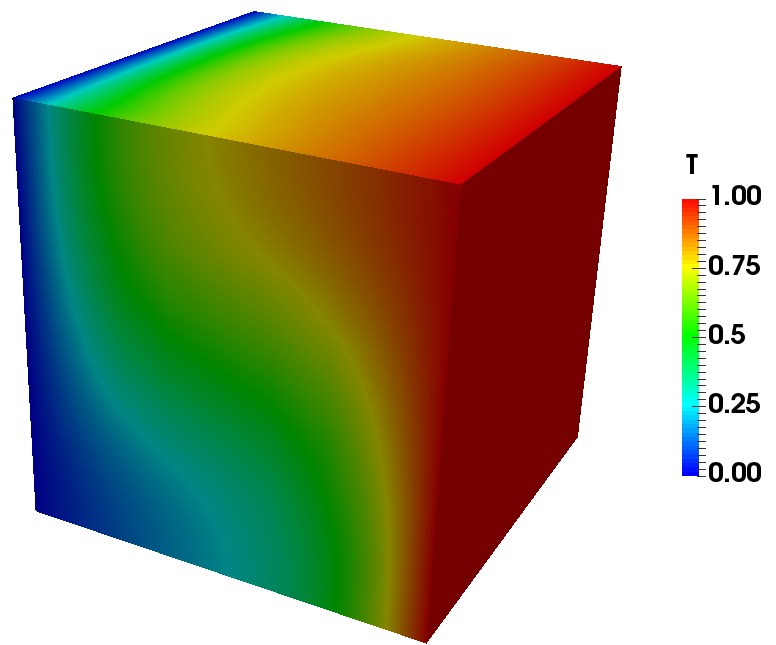} \\
(b) temperature distribution    
\end{center}
\end{minipage}
\end{center}
\end{minipage}
\caption{Three-dimensional buoyancy driven flow with $Pr = 1 $ and $Ra =10000$.} 
\label{fig:buoyancy_refinement}
\end{figure}  
\section{The buoyancy driven flow} \label{numex2}
We conclude with the non-dimensional buoyancy driven flow problem obtained 
using the Boussinesq approximation \cite{christon2002computational,elman2011fast}, 
\begin{align}
& \vec{u}\cdot \nabla \vec{u}-{Pr}^{1/2} Ra^{-1/2} \Delta \vec{u}+\nabla p+\hat{g}T=0, \label{Strong_Buoyancy_mm} \\
& \nabla \cdot \vec{u}=0, \label{Srong_Buoyancy_div} \\
& \vec{u}\cdot \nabla T - Pr^{-1/2} Ra^{-1/2} \Delta T=0, \label{Strong_Buoyancy_temp}
\end{align}
for $\Omega = [-0.5, 0.5]^3 \subset \mathbb{R}^{3}$. 
The symbols $\vec{u}$, $p$ and $T$ denote the velocity, pressure and temperature fields, respectively. 
The Prandtl number $Pr$ describes the ratio between momentum diffusivity and thermal diffusivity, while
the Rayleigh number $Ra$ describes the relation between buoyancy and viscosity within a fluid.
The quantity $\hat{g}$ is the downward unit vector along the gravity direction. 
The boundary conditions for the system \eqref{Strong_Buoyancy_mm} -- \eqref{Strong_Buoyancy_temp} are
$ \vec{u}=\vec{0} $ on $\partial \Omega$, 
$T = T_D $ on $ \partial \Omega_D $, and
$\nabla T \cdot \vec{n} = 0 $ on $ \partial \Omega_N $,
where $\partial \Omega = \partial \Omega_D \cup \partial \Omega_N$ is the whole boundary.
The subsets $\partial \Omega_D$ and $\partial \Omega_N$ are portions of $\partial \Omega$, where 
Dirichlet and Neumann boundary conditions are specified for the variable $T$, respectively.
Given the finite element spaces $\mathbf  V_{h} \subset \mathbf{H}_{0}^{1} (\Omega )$, 
$\Pi_{h} \subset L_{0}^{2} (\Omega )$ and $X_{h} \subset H^{1} (\Omega )$, 
the weak solution $(\vec{u}_{h} ,p_{h} ,T_{h}) \in \mathbf  V_{h} \times \Pi_{h} \times X_{h}$ 
of the buoyancy driven flow problem satisfies
\begin{align}
 &  (\vec{u}_h\cdot\nabla\vec{u}_h, \vec{v}_h) + 
    {Pr}^{1/2} Ra^{-1/2} (\nabla\vec{u}_h, \nabla\vec{v}_h)-(p_h, \nabla\cdot\vec{v}_h) -(T_h,v_{h,g}) = 0, \label{Weak_Buoyancy_mm}\\
 &  (\nabla\cdot\vec{u}_h, q_h)=0, \label{Weak_Buoyancy_div} \\
 &  (\nabla T_h, \nabla r_h)+ {Pr}^{-1/2} Ra^{-1/2} (\vec{u}_h\cdot\nabla T_h, r_h)=0, \label{Weak_Buoyance_temp}
\end{align}
for all $(\vec{v}_{h}, q_{h}, r_{h}) \in \mathbf V_{h} \times \Pi_{h} \times X_{h}$, 
where ${v}_{h,g} =-\vec{v}_{h} \cdot \hat{g}$. 
For the temperature variable, $T_D$ equals $1$ on the right face, $T_D$ equals $0$ on the left face, 
and Neumann zero boundary conditions are taken on the remaining boundaries.
We consider  a continuous triquadratic Lagrangian discretization for ${\vec{u}_h}$ and ${\vec{v}_h}$, a
discontinuous piecewise-linear discretization for $p_h$ and $ q_h \in \Pi_h $, and a continuous  
trilinear Lagrangian discretization for $T_h$ and ${r_h}$.
Three different values for the Rayleigh number $Ra$ are considered, namely $Ra = 1000$, $5000$ and $10000$, 
while the Prandtl is constant and equal to $1$.

The regular coarse triangulation consists of a $4 \times 4 \times 1$ hexahedral mesh, uniformly refined once
and selectively refined $1$, $2$ and $3$ times.   
The local refinement procedure follows the strategy that all elements outside 
the cylinder of radius $0.5(1-0.5^k)$ ($k=1,2,...,J-1$) and centered on the z-axis are refined.
See Figure \ref{fig:buoyancy_refinement} (a), where the z-axis is parallel to the depth direction.

We use a Newton scheme to linearize 
the nonlinear equation system \eqref{Weak_Buoyancy_mm} - \eqref{Weak_Buoyance_temp}, for more details see \cite{ke2017block}. 
At each Newton step, a GMRES solver preconditioned with our multigrid is applied to solve the linearized system. 
In the Newton scheme, the stopping tolerance for the $L_{2}-$norm of the distance between two successive solutions is $1.0\times 10^{-8}$. 
In the GMRES solver, the relative tolerance and absolute tolerance for the scaled preconditioned residual are $1.0\times 10^{-8}$ and $1.0\times 10^{-15}$, respectively.   
For the multigrid V-cycle, once again we use a Richardson smoother with one pre and post-smoothing iteration and damping factor $0.6$.
The level solver is further preconditioned with ILU.
Figure \ref{fig:buoyancy_refinement} (b) shows the temperature distribution at level $4$ for the irregular triangulation
with $Pr=1$ and $Ra=10000$. 

\begin {table}[!t]
\setlength\tabcolsep{2.5pt} 
\caption{Numerical results for both uniform and  nonuniform refinement for the buoyancy driven flow.}
\begin{center}
	\begin{tabular}{|c|c||c|c||c|c||c|c|} \hline	
	 &$Ra$  &\multicolumn{2}{|c||}{1000} & \multicolumn{2}{|c||}{5000} & \multicolumn{2}{|c|}{10000} \\ \hline
	 {L} & &uniform & \shortstack{ \\ nonuniform} & uniform & \shortstack{ \\ nonuniform} & uniform & \shortstack{ \\ nonuniform} \\ \hline
 	     2 &\shortstack{ $\,$\\Newton \\ GMRES \\ Timing}& \shortstack {5\\23.8  \\28s    } & \shortstack {5\\24   \\ 25.4s}   &  \shortstack{ 6 \\ 30.0  \\ 34.3s}   & \shortstack {6\\ 30.3 \\ 32.4s}   & \shortstack {8\\ 30.4 \\ 45.7s}   & \shortstack{ 8\\ 30.3 \\ 41.9s}   \\ \hline	 
 	      	       & $n$ &0.90 & 0.91 &0.89 & 0.93 &0.91 & 0.93\\ \hline 	       
 	     3 &\shortstack{ $\,$\\Newton \\ GMRES \\ Timing}& \shortstack {5\\40.0  \\257.6s } & \shortstack {5\\28.2 \\ 158.2s}  &  \shortstack{ 6 \\ 45.2  \\ 322.2s}  & \shortstack {6\\ 33.7 \\ 194.8s}  & \shortstack {8\\ 41.3 \\ 410.2s}  & \shortstack{ 8\\ 31.9 \\ 253.9s}  \\ \hline 
 	      	       & $n$ &0.93 & 0.93 &0.92 & 0.92 &0.91 & 0.91\\ \hline 	       
 	     4 &\shortstack{ $\,$\\Newton \\ GMRES \\ Timing}& \shortstack {5\\49.4  \\2260.2s} & \shortstack {5\\37.6 \\ 968.6s}  &  \shortstack{ 6 \\ 56.5  \\ 2916.1s} & \shortstack {6 \\43.5 \\ 1214.5s} & \shortstack {8\\ 55   \\ 3828.0s} & \shortstack{ 8\\ 40.8 \\ 1621.6s} \\ \hline 
\end{tabular} 
\end{center}
\label{buoyancy_flow}
\end{table}

The numerical results for the number of Newton iterations, 
the average number of GMRES iterations per Newton iteration, 
and the total computational time are listed in Table \ref{buoyancy_flow}.
When uniform refinement is carried out, the total number of degrees of freedom is $N = 34944$, $258044$ and $1981428$ at levels $2$, $3$ and $4$, 
respectively. With nonuniform refinement, we have $N = 30270$, $161942$ and $878726$ at levels $2$, $3$ and $4$, respectively.
At each level, the number of Newton iterations for the cases of uniform and nonuniform refinement is the same.
The average number of GMRES steps increases for both cases as the level increases.
However, for the case of non-uniformly refined meshes, it increases by a smaller factor.
Table \ref{buoyancy_flow} also shows that the computational time is slightly better than $O(\mbox{N})$ 
as the mesh is refined for both the uniform and nonuniform refinement cases. 
To estimate the complexity of the algorithm, we used the formula $ \mbox{Timing}_J = C \mbox{N}_J^n,$ 
for some constant $C$ independent of J. Then, considering 2 successive levels $J-1$ and $J$
\begin{equation} \label{time_scale}
n = \frac{ \ln(\mbox{Timing}_{J}/\mbox{Timing}_{J-1})}{\ln(\mbox{N}_{J}/\mbox{N}_{J-1})}.
\end{equation}
For $n=1$ the complexity is linear, for $n<1$ it is better than linear and for $n>1$ it is worse than linear.
In our table, $n$ is always less than 1, being consistent with the standard multigrid theory for elliptic problems.


The numerical evidence we reported shows great robustness of our multigrid algorithm.

\section{Conclusions}\label{Conclusions}
We have presented continuous basis functions for the construction of finite element spaces built on
irregular grids arising from a local midpoint refinement procedure.
We do not require the grids to be $1$-irregular as it is done in most existing works in the literature. 
This makes our results suitable for arbitrary-level hanging node configurations. 
\new{Our method works with any finite element geometry, such as quadrilaterals, 
triangles, tetrahedra, hexahedra, wedges and mixed grids, and can be applied to all Lagrangian shape functions satisfying the delta property.
Our results are comparable to existing results obtained with the deal.II library for the case of $1$-irregular grids.}
In addition, the use of continuous basis functions allows us to perform the smoothing procedure on the entire multigrid space 
rather than just on a subspace, as it is usually done for local refinement strategies.
This results in better convergence properties that improve with the number of smoothing iterations.

In conclusion, the multigrid method presented in this work is broader and more versatile compared to existing strategies \new{for 
h-refinement}.
It is robust both as a solver and as a preconditioner, and can be applied to linear and nonlinear problems 
defined on grids with arbitrary-level hanging node configurations. 

\new{Future work will consist of extending the proposed method to the case of hp-refinement.
Although only $1$-irregular meshes are implemented in the deal.II library, the case of hp-refinement is completely addressed in deal.II.
The extension of the proposed method to Hermite, H(div) and H(curl) finite elements is also object of current research.}

\bibliographystyle{plain}
\bibliography{hng_nd}
\end{document}